\documentclass[11pt, a4paper, reqno]{amsart}
\usepackage{preambolo}

\begin{document}


\setcounter{secnumdepth}{3}

\setcounter{tocdepth}{2}

\title[arXiv]{\textbf{On Bott-Chern and Aeppli cohomologies of almost complex manifolds and related spaces of harmonic forms}
}

\author[Lorenzo Sillari]{Lorenzo Sillari}

\address{Lorenzo Sillari: Scuola Internazionale Superiore di Studi Avanzati (SISSA), Via Bonomea 265, 34136 Trieste, Italy.} 
\email{lsillari@sissa.it}

\author[Adriano Tomassini]{Adriano Tomassini}

\address{Adriano Tomassini: Dipartimento di Scienze Matematiche, Fisiche e Informatiche, Unità di Matematica e
Informatica, Università degli Studi di Parma, Parco Area delle Scienze 53/A, 43124, Parma, Italy}
\email{adriano.tomassini@unipr.it}

\maketitle

\begin{abstract} 
\noindent \textsc{Abstract}. In this paper we introduce several new cohomologies of almost complex manifolds, among which stands a generalization of Bott-Chern and Aeppli cohomologies defined using the operators $d$, $d^c$. We explain how they are connected to already existing cohomologies of almost complex manifolds and we study the spaces of harmonic forms associated to $d$, $d^c$, showing their relation with Bott-Chern and Aeppli cohomologies and to other well-studied spaces of harmonic forms. Notably, Bott-Chern cohomology of $1$-forms is finite-dimensional on compact manifolds and provides an almost complex invariant $h^1_{d + d^c}$ that distinguishes between almost complex structures. On almost K\"ahler $4$-manifolds, the spaces of harmonic forms we consider are particularly well-behaved and are linked to harmonic forms considered by Tseng and Yau in the study of symplectic cohomology.

\end{abstract}

\blfootnote{  \hspace{-0.55cm} 
{\scriptsize 2020 \textit{Mathematics Subject Classification}. Primary: 32Q60, 58A14; Secondary: 53C15, 58J05. 

\noindent The authors are partially supported by GNSAGA of INdAM. The second author is partially
supported by the Project PRIN 2017 “Real and Complex Manifolds: Topology, Geometry and
holomorphic dynamics”.}}

\section{Introduction}\label{sec:intro}

Let $M$ be a compact complex manifold. The exterior derivative decomposes on the space of $(p,q)$-forms on $M$ as $d=\partial+\bar\partial$ and the operator $\bar\partial$ satisfies $\bar\partial^2=0$. Then the {\em Dolbeault cohomology groups $H_{\bar\partial}^{\bullet,\bullet}(M)$} of the complex manifold $M$ are defined as the cohomology groups associated to the operator $\bar\partial$. It turns out that $H_{\bar\partial}^{\bullet,\bullet}(M)$ are finite dimensional complex vector spaces and consequently their dimensions $h_{\bar\partial}^{\bullet,\bullet}$ provide natural complex invariants of $M$. Other natural complex cohomologies are the {\em Bott-Chern} and {\em Aeppli} cohomology groups \cite{BC65, Aep65}.
They have been widely studied on complex manifolds (see, for instance, \cite{Sch07, LUV14, TY14, ADT16, AK17, PSU21, Big69, Big70}) and they are defined starting from the double complex $(A^{\bullet, \bullet}, d, d^c)$ (or, equivalently, $(A^{\bullet, \bullet}, \partial, \bar \partial)$) as
\[
H^{p,q}_{BC} = \frac{\ker d \cap \ker d^c \cap A^{p,q}}{\Ima dd^c \cap A^{p,q}} \quad \text{and} \quad H^{p,q}_{A} = \frac{\ker d d^c \cap A^{p,q}}{(\Ima d + \Ima d^c) \cap A^{p,q}}.
\]
On the one hand, they are related to complex de Rham cohomology, $d^c$-cohomology, and Dolbeault cohomology. On the other hand, if we fix a Riemannian metric, the spaces of harmonic forms associated to Bott-Chern and Aeppli Laplacians are isomorphic to the corresponding cohomology by Hodge theory \cite{Sch07}. In particular, Bott-Chern and Aeppli cohomologies are finite dimensional and they can be used to detect $dd^c$-manifolds \cite{AT13}.
\vspace{.2cm}

On compact symplectic manifolds, Tseng and Yau (\cite{TY12a, TY12b}, see also \cite{TTY16}) introduced a symplectic counterpart of Bott-Chern and Aeppli cohomologies defined starting from the double complex $(A^\bullet, d, d^\Lambda)$ as 
\[
H^{k}_{d + d^\Lambda} = \frac{\ker d \cap \ker d^\Lambda \cap A^{k}}{\Ima dd^\Lambda \cap A^{k}} \quad \text{and} \quad H^{k}_{dd^\Lambda} = \frac{\ker d d^\Lambda \cap A^{k}}{(\Ima d + \Ima d^\Lambda) \cap A^{k}}.
\]
These cohomologies admit a decomposition in terms of primitive forms and Hodge theory allows to deduce that they are finite dimensional, in complete analogy with the complex case. An equivalent definition can be given in terms of the operators $\partial_+$, $\partial_-$, that should be thought as a generalization of $\partial$, $\bar \partial$ to symplectic manifolds \cite{TY12b}.
\vspace{.2cm}

On almost complex manifolds, the usual definition of Bott-Chern and Aeppli cohomologies does not extend to any of the most studied operators (e.g., $d$, $d^c$, $\partial$, $\bar \partial$, $\mu$, $\bar \mu$, $\delta$, $\bar \delta$) since there is no natural double complex to consider. Motivated and inspired by the complex Bott-Chern and Aeppli cohomologies and the symplectic cohomologies of Tseng and Yau, in this paper we give a definition of Bott-Chern and Aeppli cohomologies for almost complex manifolds that is based on $d$, $d^c$ (rather than $\partial$, $\bar \partial$) and is defined on suitable complexes of forms $B^\bullet$ and $C^\bullet$ as
 \[
H^k_{d + d^c}:= \frac{\ker (d \colon B^k \rightarrow B^{k+1}) \cap \ker (d^c \colon B^k \rightarrow B^{k+1})}{ \Ima (d d^c \colon B^{k-2} \rightarrow B^k)},
\]
and
\[
H^k_{dd^c} := \frac{\ker (d d^c \colon C^k \rightarrow C^{k+2}) }{ \Ima (d \colon C^{k-1} \rightarrow C^k) + \Ima (d^c \colon C^{k-1} \rightarrow C^k)}.
\]

It turns out that in the definition of our cohomologies, we could have equivalently considered the operators $\delta$, $\bar \delta$ \cite{BT01, TT20}. Therefore, $\delta$, $\bar \delta$ appear to be an appropriate generalization of $\partial$, $\bar \partial$ to almost complex manifolds when studying Bott-Chern and Aeppli cohomologies. 

Bott-Chern and Aeppli cohomologies are naturally related to other cohomologies of almost complex manifolds. First of all, they have a strong connection with de Rham and $d^c$-cohomologies. Moreover, the $J$-invariant and $J$-anti-invariant cohomologies introduced by Li and Zhang (\cite{LZ09}, see also \cite{DLZ10, DLZ13, AT11}) can be generalized to \emph{$J$-even} and \emph{$J$-odd cohomologies} that naturally induce a $\Z_2$-splitting of Bott-Chern cohomology.
\vspace{.2cm}

Hodge theory for almost complex manifolds has been developed studying several Laplacians related to complex operators (among the others, see \cite{CW20b, TT20, PTa22, PT22, HZ22}), including $\delta$, $\bar \delta$. However, this is the first time in the literature that one considers $d$, $d^c$. Thus, we proceed to study Hodge theory of the operators $d$, $d^c$ and to understand to which extent it can be used to describe Bott-Chern and Aeppli cohomologies. Even though Bott-Chern and Aeppli cohomologies can be defined either in terms of $d$, $d^c$ or $\delta$, $\bar \delta$, when studying harmonic forms the two approaches are not equivalent (see \cite{TT20} for the second approach).\\
In the almost complex case, there are four natural Bott-Chern and Aeppli Laplacians arising from $d$, $d^c$ to consider, whose associated spaces of harmonic forms are finite-dimensional and isomorphic to each other. In general, such spaces of harmonic forms are not isomorphic to Bott-Chern and Aeppli cohomologies, that might be infinite dimensional even on compact manifolds. Nevertheless, focusing on $1$-forms and $(2m-1)$-forms, there is an isomorphism between harmonic forms and the cohomologies. As a consequence, $H^1_{d + d^c}$ is isomorphic to $H^{2m-1}_{d d^c}$ and it is a finite-dimensional vector space whose dimension $h^1_{d + d^c}$ provides an invariant of almost complex manifolds. The number $h^1_{d + d^c}$ allows to distinguish between almost complex structures even when other invariants fail in doing so. More precisely, it distinguishes between almost complex structures whose Nijenhuis tensors have the same rank and between almost K\"ahler structures compatible with the same symplectic form.

On compact almost K\"ahler $4$-manifolds, our spaces of harmonic forms have a particularly nice behaviour and show a close relation with harmonic forms built by Tseng and Yau using the operators $d$, $d^\Lambda$.
\vspace{.2cm}

Our definition of Bott-Chern and Aeppli cohomologies should be compared with a related object recently introduced by Coelho, Placini and Stelzig \cite{CPS22}. The two objects are substantially different. For instance, we lose the $(\Z \times \Z)$-grading on the cohomology and therefore also the connection with Dolbeault cohomology of almost complex manifolds defined by Cirici and Wilson (\cite{CW21}, see also \cite{CW22}). In exchange, we gain a strong link with other almost complex cohomologies, as we explained above.
\vspace{.2cm}

The objects introduced in this paper enlarge the list of cohomologies available for the study of almost complex manifolds and underline connections among them (see also \cite{Mus86, Mus21, ST21, CGGH21, CGG22}).
\vspace{.2cm}

The paper is organized as follows: in section \ref{sec:prel} we fix the notation and we recall some preliminary facts.\\
Section \ref{sec:cohom:interpretation} is devoted to the study of almost complex cohomologies: in subsection \ref{sec:BCA:cohom} we give the definition of Bott-Chern and Aeppli cohomologies for the operators $d$, $d^c$ (Definition \ref{def:BCA:cohomology}). In analogy with the complex case, we prove that Aeppli cohomology is a module over Bott-Chern cohomology (Proposition \ref{prop:module}), that there are maps linking them to de Rham and $d^c$-cohomologies (Proposition \ref{prop:diagram}) and that they are a well-defined invariant of almost complex manifolds (Theorem \ref{thm:acinvariants}). In subsection \ref{sec:delta:cohom}, we study the cohomologies of the operators $\delta$, $\bar \delta$. We define their cohomology and prove that their Bott-Chern and Aeppli cohomologies are isomorphic to the ones defined for $d$, $d^c$ (Proposition \ref{prop:BCA:equivalence}). In subsection \ref{sec:even:odd} we study $J$-even and $J$-odd cohomologies and prove that Bott-Chern cohomology admits a natural splitting according to the parity of forms (Theorem \ref{thm:BC:splitting} and Corollary \ref{cor:BC:splitting}).\\
Section \ref{sec:hodge:theory} contains Hodge Theory for $d$, $d^c$: in subsection \ref{sec:BCA:laplacians} we study the Laplacian $\Delta_{d + d^c}$ and the spaces of $(d + d^c)$-harmonic forms $\H^k_{d + d^c}$. On compact manifolds, the dimensions of such spaces $h^k_{d + d^c}$ provide an almost Hermitian invariant (Theorem \ref{thm:ah:invariants}). In subsection \ref{sec:ac:hodge:numbers} and \ref{sec:cohom:iso} we study the relation between $H^k_{d + d^c}$ and $\H^k_{d + d^c}$ and to which extent the numbers $h^k_{d + d^c}$ are metric independent. It turns out that $H^1_{d + d^c}$ is isomorphic to $\H^1_{d + d^c}$ (Proposition \ref{prop:BC:harmonic:inclusion}). Therefore, the number $h^1_{d + d^c}$ is an almost complex invariant that allows to distinguish between almost complex structures whose Nijenhuis tensors have the same rank (Proposition \ref{prop:h1:same:rank}). In subsection \ref{sec:delta:harmonic} we underline the relations between harmonic forms for $d$, $d^c$ and $\delta$, $\bar \delta$.\\
In section \ref{sec:ak:4manifolds} we focus on compact almost K\"ahler $4$-manifolds. We prove a decomposition for the space of $(d + d^c)$ harmonic $2$-forms (Theorem \ref{thm:decomposition}) and we show that the space of $(d + d^c)$-harmonic $3$-forms coincides with the space of $(d + d^\Lambda)$-harmonic $3$-forms considered by Tseng and Yau (Theorem \ref{thm:h3}). The number $h^1_{d +d^c}$ allows to distinguish between almost complex structures compatible with the same symplectic form (Proposition \ref{prop:symplectic}).\\
Finally, section \ref{sec:examples} contains several examples of computations of Bott-Chern cohomology and of the spaces of $(d + d^c)$-harmonic forms.
\vspace{.2cm}

This paper is part of the ongoing PhD project of the first author, focused on the study of invariants of almost complex and almost symplectic manifolds. The study of the objects introduced in this paper continues in several works in preparation.

\subsection*{Acknowledgments.} The authors kindly thank Tom Holt and Riccardo Piovani for useful discussions. They are also grateful to Jonas Stelzig for useful comments and remarks that improved the presentation and for pointing out a missing assumption in theorem \ref{thm:ah:invariants}. The first author is grateful to the Department of Mathematical, Physical and Computer Sciences of the University of Parma for its warm hospitality.

\section{Preliminaries}\label{sec:prel}
In this section we fix the notation and we recall definitions and facts that will be used in the rest of the paper.
\vspace{.2cm}

Let $(M,J)$ be an almost complex $2m$-manifold, that is a $2m$-dimensional smooth manifold endowed with a smooth $(1,1)$-tensor field $J$ satisfying $J^2=-\Id$. Denote by $A^k$ the space of smooth, complex $k$-forms on $M$. The almost complex structure $J$ induces a bigrading 
\[
A^k = \bigoplus_{p+q = k} A^{p,q}
\]
on $k$-forms and it acts by duality on $A^{p,q}$ preserving the bigrading. On $(p,q)$-forms, $J^2 = (-1)^{p+q} \Id$. With respect to the bigrading, the differential $d$ decomposes as $d = \mu + \partial + \bar \partial + \bar \mu$. From the equation $d^2=0$, we have the relations
\begin{equation}\label{eq:d2}
\begin{cases}
    \bar \mu^2 = 0, \\
    \bar \mu \bar \partial + \bar \partial \bar \mu =0, \\
    \bar \partial^2 + \bar \mu \partial + \partial \bar \mu =0, \\
    \partial \bar \partial + \bar \partial \partial + \mu \bar \mu + \bar \mu \mu =0,
\end{cases}
\end{equation}
and the conjugate equations. According to the celebrated Theorem of Newlander and Nirenberg, the almost complex structure $J$ is integrable, i.e., it is induced by the structure of a complex manifold, if and only if the Nijenhuis tensor $N_J$ of $J$ defined as
$$
N_J(X,Y) := [JX,JY]-J[JX,Y]-J[X,JY]-[X,Y]
$$
vanishes. The latter fact it turns to be equivalent to say that $\bar \mu = 0$. Consider the operator $d^c := J^{-1} d J$. On a complex manifold, we have that
\[
d^2 =0, \quad (d^c)^2 =0, \quad d d^c + d^c d =0,
\]
and, consequently, $(A^\bullet, d, d^c)$ is a double complex. One can naturally define four cohomologies associated to such a complex, namely the \emph{de Rham cohomology} $H^k_{d}$, the \emph{$d^c$-cohomology}
\[
H^k_{d^c} := \frac{\ker d^c \cap A^k}{\Ima d^c \cap A^k},
\]
the \emph{Bott-Chern cohomology}
\[
H^{p,q}_{BC} := \frac{\ker d \cap \ker d^c \cap A^{p,q}}{\Ima d d^c \cap A^{p,q}},
\]
and the \emph{Aeppli cohomology}
\[
H^{p,q}_{A} := \frac{\ker d d^c \cap A^{p,q}}{(\Ima d + \Ima d^c) \cap A^{p,q}}.
\]
Furthermore, these cohomologies can be arranged in the following diagram

\begin{equation}\label{complex:diagram}
\begin{tikzcd}
                                       & {H^{p,q}_{BC}} \arrow[ld] \arrow[rd] &                           \\
H^{p+q}_d \arrow[rd] \arrow[rr, leftrightarrow, "J"', "\sim"] &                                        & H^{p+q}_{d^c} \arrow[ld] \\
                                       & {H^{p,q}_A}                            &                          
\end{tikzcd}
\tag{D}
\end{equation}
where $J$ provides an isomorphism between de Rham and $d^c$-cohomologies. Using the decomposition $d = \partial + \bar \partial$, one can write the real operators $d$, $d^c$ in terms of the complex operators $\partial$, $\bar \partial$ and vice versa:
\begin{equation}\label{eq:ddc:deldelbar}
d = \partial + \bar \partial, \quad d^c = i (\bar \partial - \partial), \quad \partial = \frac{1}{2} (d + id^c), \quad \bar \partial = \frac{1}{2} (d - id^c).
\end{equation}
Replacing $d$, $d^c$ by $\partial$, $\bar \partial$, it is possible to define Dolbeault cohomology and to obtain an equivalent definition of Bott-Chern and Aeppli cohomologies.

If $J$ is not integrable, there are two natural operators built using $d$ and $d^c$ that in some sense generalize $\partial$ and $\bar \partial$ to the almost complex case, namely
\begin{equation}\label{eq:delta:ddc}
\delta := \frac{1}{2} (d + i d^c) = \partial + \bar \mu \quad \text{and} \quad \bar \delta := \frac{1}{2} (d - i d^c) = \bar \partial + \mu.
\end{equation}
They induce a decomposition of $d$ and $d^c$
\begin{equation}\label{eq:ddc:delta}
d = \delta + \bar \delta, \quad d^c = i( \bar \delta - \delta).
\end{equation}
and, by \eqref{eq:d2}, \eqref{eq:delta:ddc}, and \eqref{eq:ddc:delta}, we have that
\begin{equation}\label{eq:delta2}
\delta^2 = - \bar \delta^2 = \frac{i}{4} (d d^c + d^c d) =  \partial^2 - \bar \partial^2,
\end{equation}
and
\begin{equation}\label{eq:deltabardelta}
    \delta \bar \delta = - \frac{i}{4} (d d^c - d^c d)
\end{equation}
The condition $d d^c + d^c d =0$ is equivalent to integrability of $J$. Thus, when $J$ is not integrable, it is still true that $d^2 =0$ and $(d^c)^2=0$, but $d d^c + d^c d \neq 0$. Similarly, we have that $\delta \bar \delta + \bar \delta \delta =0$, but $\delta^2 \neq 0$, $\bar \delta ^2 \neq 0$. In particular, $(A^\bullet, d, d^c)$ and $(A^\bullet, \delta, \bar \delta)$ are not double complexes and their Bott-Chern and Aeppli cohomologies are not well-defined. Furthermore, since $\delta^2 \neq 0$, the operator $\delta$ lacks an associated cohomology theory.
\vspace{.2cm}

We consider a smooth map between almost complex manifolds compatible with almost complex structures.\\
Let $(M_1,J_1)$ and $(M_2,J_2)$ be two almost complex manifolds. A smooth map $f \colon M_1 \rightarrow M_2$ is said to be \emph{pseudo-holomorphic} if
\begin{equation}\label{eq:J:holomorphic}
    df \circ J_1 = J_2 \circ df.
\end{equation}
\vspace{.2cm}

We now recall some basic facts regarding almost Hermitian manifolds.\\
Let $(M,J)$ be an almost complex manifold and let $g$ be a $J$-compatible metric, i.e., a Riemannian metric on $M$ such that 
\[
g(J \cdot, J \cdot) = g (\cdot, \cdot).
\]
We say that $(M,J,g)$ is an \emph{almost Hermitian manifold}. Consider the $\C$-linear Hodge $*$ operator
\[
* \colon A^{p,q} \longrightarrow A^{m-q,m-p}.
\]
Assume now that $M$ is compact. On a compact almost Hermitian manifold there is a natural Hermitian product induced on $A^k$ by $g$, defined on forms $\alpha$ and $\beta$ by 
\[
\langle \alpha, \beta \rangle := \int_M \alpha \wedge \overline{*\beta}.
\]

For any differential operator $P \colon A^\bullet \rightarrow A^\bullet$, we denote by $P^*$ its adjoint operator, given by $P^* := - * P *$. If we consider separately the operators $d$ and $d^c$, we have the associated \emph{Hodge Laplacians}
\[
\Delta_d = d d^* + d^* d \quad \text{and} \quad \Delta_{d^c} = d^c (d^c)^* + (d^c)^* d^c.
\]
These are $2^\text{nd}$-order, self-adjoint, elliptic operators, whose kernel is finite dimensional and it is isomorphic to de Rham cohomology, resp.\ $d^c$-cohomology, by Hodge theory.\\
If $J$ is also integrable, there are two natural Laplacians arising from the combination of $d$ and $d^c$, namely the \emph{Bott-Chern Laplacian}
\[
\Delta_{BC} = d d^c ( d d^c)^* + ( d d^c)^* d d^c + d^* d^c ( d^* d^c )^* + ( d^* d^c )^* d^* d^c + d^* d + (d^c)^* d^c,
\]
and the \emph{Aeppli Laplacian}
\[
\Delta_{A} = d d^c ( d d^c)^* + ( d d^c)^* d d^c + d (d^c)^* ( d (d^c)^* )^* + ( d (d^c)^* )^* d (d^c)^*  + d d^*  +  d^c (d^c)^*.
\]
These are $4^\text{th}$-order, self-adjoint, elliptic operators, whose kernel is finite dimensional and it is isomorphic to Bott-Chern cohomology, resp.\ Aeppli cohomology. As a consequence, Bott-Chern and Aeppli cohomologies of a compact complex manifold are finite dimensional.
\vspace{.2cm}

Tseng and Yau defined a symplectic counterpart to Bott-Chern and Aeppli cohomologies of complex manifolds \cite{TY12a}, built using the operators $d$, $d^\Lambda$ instead of $d$, $d^c$. We recall here some basic facts that lead to the construction of those cohomologies.\\
Let $(M,\omega)$ be a symplectic manifold. The \emph{Lefschetz operator} $L$ is the operator
\begin{align*}
    L \colon A^k &\longrightarrow A^{k+2}, \\
    \alpha &\longmapsto \omega \wedge \alpha.
\end{align*}
The \emph{dual Lefschetz operator} $\Lambda \colon A^k \rightarrow A^{k-2}$ is defined as contraction by $\omega$. A $k$-forms $\alpha$ is called \emph{primitive} if $\Lambda \alpha =0$. The triple $(L, \Lambda, H:= [ \Lambda , L])$ defines a representation of $\mathfrak{sl} (2, \C)$ acting on $A^\bullet$ that induces the \emph{Lefschetz decomposition} of $\alpha \in A^k$ \cite{Wei58}, namely
\[
\alpha = \sum_{j \ge \max \{ k-m,0 \}}^{} L^j P^{k -2j},
\]
where $P^{k-2j}$ are primitive forms. There is a natural symplectic operator defined as
\[
d^\Lambda := [ d, \Lambda] = d \Lambda - \Lambda d.
\]
Using $d$, $d^\Lambda$, one can build symplectic cohomologies 
\[
H^{k}_{d + d^\Lambda} := \frac{\ker d \cap \ker d^\Lambda \cap A^{k}}{\Ima d d^\Lambda \cap A^{k}}, 
\]
and
\[
H^{k}_{d d^\Lambda} := \frac{\ker d d^\Lambda \cap A^{k}}{(\Ima d + \Ima d^\Lambda) \cap A^{k}}.
\]
Fix an $\omega$-compatible almost complex structure $J$, and let $g$ be the Riemannian metric
\[
g (\cdot, \cdot) = \omega (J \cdot, \cdot).
\]
We say that $(M, J, g, \omega)$ is an \emph{almost K\"ahler manifold}. The compatible triple $(J,g,\omega)$ allows to express $d^\Lambda$ in function of $d$, $J$ and $g$ as
\[
d^\Lambda = (d^c)^*.
\]
If $M$ is compact, one can develop Hodge theory for the \emph{$(d + d^\Lambda)$-Laplacian}
\[
\Delta_{d + d^\Lambda} = d d^\Lambda ( d d^\Lambda)^* + ( d d^\Lambda)^* d d^\Lambda + d^* d^\Lambda ( d^* d^\Lambda )^* + ( d^* d^\Lambda )^* d^* d^\Lambda + d^* d + (d^\lambda)^* d^\Lambda,
\]
and the \emph{$d d^\Lambda$-Laplacian}
\[
\Delta_{d d^\Lambda} = d d^\Lambda ( d d^\Lambda)^* + ( d d^\Lambda)^* d d^\Lambda + d (d^\Lambda)^* ( d (d^\Lambda)^* )^* + ( d (d^\Lambda)^* )^* d (d^\Lambda)^*  + d d^*  +  d^\Lambda (d^\Lambda)^*,
\]
obtaining isomorphisms
\[
H^{k}_{d + d^\Lambda} \cong \H^{k}_{d + d^\Lambda} \quad \text{and} \quad H^{k}_{d d^\Lambda} \cong \H^{k}_{d d^\Lambda},
\]
where $\H^{k}_{d + d^\Lambda}$, resp.\ $\H^{k}_{d d^\Lambda}$, is the space of \emph{$(d + d^\Lambda)$-harmonic forms}, resp.\ \emph{$d d^\Lambda$-harmonic forms}. In particular, $(d + d^\Lambda)$ and $d d^\Lambda$-cohomologies of a compact symplectic manifold are finite dimensional.

\section{Bott-Chern and Aeppli cohomologies of almost complex manifolds}\label{sec:cohom:interpretation}

In this section we introduce three notions of cohomologies for almost complex manifolds. We begin with a definition of Bott-Chern and Aeppli cohomologies given using the operators $d$, $d^c$ and we establish their basic properties. We observe that the spaces of forms where Bott-Chern and Aeppli cohomologies are defined also appear as natural spaces where one can define the cohomology of $\delta$, $\bar \delta$ and the associated Bott-Chern and Aeppli cohomologies. It turns out that Bott-Chern and Aeppli cohomologies for $d$, $d^c$ and for $\delta$, $\bar \delta$ coincide. Finally, we introduce $J$-even and $J$-odd cohomologies that appear as a special case of the cohomologies $H^{S}_J$ \cite{AT11}, and are closely related to $J$-invariant and $J$-anti-invariant cohomologies of Draghici, Li and Zhang \cite{LZ09, DLZ10, DLZ13}. Bott-Chern cohomology naturally admits a $\Z_2$-splitting induced by such cohomologies.

\subsection{Bott-Chern and Aeppli cohomologies of the operators \texorpdfstring{$d$}{}, \texorpdfstring{$d^c$}{}.}\label{sec:BCA:cohom}

On an almost complex manifold $(M,J)$, with non-integrable $J$, we have that $d d^c + d^c d \neq 0$ and $(A^\bullet, d, d^c)$ is no longer a double complex. To define Bott-Chern and Aeppli cohomologies, we introduce a subcomplex and a quotient complex of $A^\bullet$ (cf. \cite{CPS22}) on which $d$ and $d^c$ anti-commute.

\begin{definition}\label{def:double:complex}
Consider the subcomplex of $A^\bullet$ given by
\[
B^\bullet := \{ \alpha \in A^\bullet : (d d^c + d^c d) \alpha =0 \}
\]
and the quotient complex of $A^\bullet$ given by
\[
C^\bullet := \frac{A^\bullet}{\Ima (d d^c + d^c d)}.
\]
\end{definition}

\begin{proposition}\label{prop:double:complex}
The complexes $(B^\bullet, d, d^c)$ and $(C^\bullet, d, d^c)$ are $\Z$-graded double complexes.
\end{proposition}

\begin{proof}
If $\alpha \in B^k$, then $d \alpha \in B^{k+1}$. Indeed,
\[
(d d^c + d^c d)d \alpha = d d^c d \alpha = - d ( d d^c) \alpha = 0.
\]
Similarly, if $\alpha \in B^k$ then $d^c \alpha \in B^{k+1}$. Finally, we have that $d d^c + d^c d=0$ on $B^\bullet$ by definition. The proof for $(C^\bullet, d, d^c)$ is similar.
\end{proof}

We now define Bott-Chern and Aeppli cohomologies for almost complex manifolds.

\begin{definition}\label{def:BCA:cohomology}
    Let $(M,J)$ be an almost complex manifold. The \emph{Bott-Chern cohomology of $(M,J)$} is the Bott-Chern cohomology of $(B^\bullet, d, d^c)$, i.e.,
    \[
    H^k_{d + d^c} (M,J) := \frac{\ker (d \colon B^k \rightarrow B^{k+1}) \cap \ker (d^c \colon B^k \rightarrow B^{k+1})}{ \Ima (d d^c \colon B^{k-2} \rightarrow B^k)}.
    \]
    The \emph{Aeppli cohomology of $(M,J)$} is the Aeppli cohomology of $(C^\bullet, d, d^c)$, i.e.,
    \[
    H^k_{dd^c} (M,J) := \frac{\ker (d d^c \colon C^k \rightarrow C^{k+2}) }{ \Ima (d \colon C^{k-1} \rightarrow C^k) + \Ima (d^c \colon C^{k-1} \rightarrow C^k)}.
    \]
    If there is no ambiguity, we will simply write $H^k_{d + d^c}$ and $H^k_{dd^c}$, omitting the underlying manifold and the almost complex structure.
\end{definition}

\begin{remark}[Notation]\label{rem:notation}
    We choose to denote Bott-Chern and Aeppli cohomologies of $(M,J)$ by $H^k_{d + d^c}$, resp. $H^k_{d d^c}$, instead of adopting $H^k_{BC}$, resp. $H^k_{A}$ for three reasons: first, in this way we stress the fact that such cohomologies are built out of the operators $d, d^c$ instead of the operators $\partial, \bar \partial$. Second, we avoid confusion with the Bott-Chern and Aeppli cohomologies of almost complex manifolds introduced by Coelho, Placini and Stelzig \cite{CPS22}. Third, we underline the conceptual link with the symplectic cohomologies $H^k_{d + d^\Lambda}$, resp. $H^k_{d d^\Lambda}$, introduced by Tseng and Yau \cite{TY12a}.
\end{remark}

\begin{remark}
In the complex case, $d d^c + d^c d=0$ on all forms, so that $B^\bullet = A^\bullet$, $C^\bullet = A^\bullet$, and we recover the usual definition of Bott-Chern and Aeppli cohomologies.
\end{remark}

In the complex case, Aeppli cohomology is a module over Bott-Chern cohomology. This is true even when $J$ is not integrable.

\begin{proposition}\label{prop:module}
    Aeppli cohomology has the structure of module over Bott-Chern cohomology. In particular, there is a well-defined pairing 
    \begin{align*}
    H^k_{d + d^c} \times H^\ell_{d d^c} &\longrightarrow H^{k+\ell}_{d d^c}, \\
    ( [\alpha]_{d + d^c} , [\gamma]_{d d^c} ) &\longmapsto [ \alpha \wedge \gamma]_{d d^c}.
    \end{align*}
\end{proposition}

\begin{proof}
Let $[ \alpha ]_{d+d^c} \in H^k_{d + d^c}$. Then 
\[
[ \alpha ]_{d+d^c} = \alpha + d d^c \beta,
\]
where $d \alpha =0$, $d^c \alpha =0$, and $\beta \in B^{k-2}$. Pick any $[\gamma]_{d d^c} \in H^\ell_{d d^c}$. Then
\[
[\gamma]_{d d^c} = \gamma + d \eta + d^c \theta,
\]
up to $(d d^c + d^c d)$-exact forms, where $d d^c \gamma = [0]_{C^\bullet}$. To prove the proposition it is enough to check that $[\alpha \wedge \gamma]_{d d^c} \in H^{k+\ell}_{d d^c}$ is well-defined. We have that $d d^c (\alpha \wedge \gamma) = [0]_{C^\bullet}$ since $\alpha$ is $d$-closed and $d^c$-closed, and $\gamma $ is $d d^c$-closed. Moreover, the cohomology class does not depend on the representatives. Indeed, we have that
{\small
\[
(\alpha + d d^c \beta) \wedge (\gamma + d \eta + d^c \theta) = \alpha \wedge \gamma + \alpha \wedge d \eta + \alpha \wedge d^c \theta + dd^c \beta \wedge \gamma + d d^c \beta \wedge d \eta + d d^c \beta \wedge d^c \theta.
\]
}

When we pass to cohomology, the second, third, fifth and sixth terms on the right hand side vanish because they are $d$-exact or $d^c$-exact. The fourth term is both $d$ and $d^c$-exact, since when computing Aeppli cohomology, we are considering classes in $C^\bullet$.
\end{proof}

\begin{proposition}\label{prop:diagram}
    There is a diagram of cohomologies
    
\begin{equation}\label{almost:complex:diagram}
\begin{tikzcd}
                                       & {H^k_{d+ d^c}} \arrow[ld] \arrow[rd] &                           \\
H^k_d \arrow[rd] \arrow[rr, leftrightarrow, "J"', "\sim"] &                                        & H^k_{d^c} \arrow[ld] \\
                                       & {H^k_{d d^c}}                            &                          
\end{tikzcd}
\tag{D$'$}
\end{equation}

that generalizes diagram \eqref{complex:diagram} to the case of almost complex manifolds.
\end{proposition}

\begin{proof}
    Let $[ \alpha ]_{d+d^c} \in H^k_{d + d^c}$. Then $[\alpha]_{d + d^c} = \alpha + d d^c \beta$. On the one side, $d \alpha=0$ and $d d^c \beta$ is $d$-exact, so that $\alpha$ defines a de Rham cohomology class. On the other side, $d^c \alpha =0$ and $d d^c \beta = - d^c d \beta$, so it also defines a $d^c$-cohomology class. Thus the morphisms from Bott-Chern cohomology to de Rham and $d^c$-cohomologies are given by the identity on representatives. Let now $ [\gamma]_{d} \in H^k_d$. Then
    \[
    d d^c \gamma= (d d^c + d^c d) \gamma, 
    \]
    that is the zero class in $C^\bullet$, so that $\gamma$ defines an Aeppli cohomology class. Changing representative in the de Rham cohomology class modifies $\gamma$ by a $d$-exact form, preserving the corresponding Aeppli cohomology class. The same proof holds for a $d^c$-cohomology class. Thus the morphisms from de Rham and $d^c$-cohomologies to Aeppli cohomology are given by the identity on representatives composed with the projection on $C^\bullet$.
\end{proof}

\begin{remark}
    The morphisms going from Bott-Chern cohomology to de Rham and $d^c$-cohomologies are not injective nor surjective. The same holds for those going from de Rham and $d^c$-cohomologies to Aeppli cohomology. This is true even if $J$ is integrable (see, for example, \cite{Ang14}).
\end{remark}

Next proposition establishes that Bott-Chern and Aeppli cohomologies are almost complex invariants.

\begin{theorem}\label{thm:acinvariants}
    Let $(M,J)$ and $(M', J')$ be almost complex manifolds and let $f \colon M \rightarrow M'$ be a pseudo-holomorphic map. Then $f$ induces a morphism of differential $\Z$-graded algebras
    \begin{equation*}
        f^* \colon H^\bullet_{d + d^c} (M', J') \longrightarrow H^\bullet_{d + d^c} (M, J),
    \end{equation*}
    and a morphism of differential $\Z$-graded modules over $H^\bullet_{d +d^c}$
    \begin{equation*}
        f^* \colon H^\bullet_{d d^c} (M', J') \longrightarrow H^\bullet_{d d^c} (M, J).
    \end{equation*}
    Furthermore, if $f$ is also a diffeomorphism, then $f^*$ is an isomorphism.
\end{theorem}

\begin{proof}
    The pullback $f^*$ commutes with $d$. By \eqref{eq:J:holomorphic}, it also commutes with $d^c$. In particular, $f^*$ sends the double complexes $(B^\bullet, d, d^c)$, $(C^\bullet, d, d^c)$ defined for $(M',J')$ into the same objects defined for $(M,J)$. As a consequence, it defines a morphism on the cohomologies.\\
    If $f$ is also a diffeomorphism, then $M = M'$ and $df$ is an isomorphism. Thus $J = df^{-1} \circ  J' \circ df$ and $f^*$ is an isomorphism with inverse $(f^{-1})^*$.
\end{proof}

\begin{remark}\label{rem:bigraded}
    In principle, one could try to define a bigraded version of Bott-Chern and Aeppli cohomologies for $d$, $d^c$. However, the operator $d d^c$ does not preserve the bigrading and giving such a definition would require to impose further conditions that would make the construction less natural. Nevertheless, if $(p,q) \in \{(1,0), (0,1), (2,0), (0,2) \}$, one can still give a bigraded notion of Bott-Chern cohomology setting
    \[
    H^{p,q}_{d + d^c} := \ker d \cap A^{p,q},
    \]
    since $d$-closed $(p,q)$-forms coincide with $d^c$-closed $(p,q)$-forms, while for $(p,q)=(1,1)$ one can set
    \[
    H^{1,1}_{d + d^c} := \frac{\ker d \cap A^{1,1}}{\Ima ( d d^c \colon \ker (dd^c + d^c d) \cap C^{\infty} (M)  \rightarrow \ker (d d^c + d^c d) \cap A^{1,1})}.
    \]
\end{remark}

Using the bigraded cohomology groups we just defined, we can prove a bigraded splitting in the case of $1$-forms.

\begin{proposition}\label{prop:decomposition:1forms}
    Let $(M,J)$ be an almost complex manifold. Then
    \begin{equation}\label{BC:cohom:decomposition:h1}
        H^1_{d + d^c} = H^{1,0}_{d + d^c} \oplus H^{0,1}_{d + d^c}.
    \end{equation}
\end{proposition}

\begin{proof}
    A Bott-Chern cohomology class in $H^1_{d + d^c}$ is given by a $1$-form $\alpha$ that is $d$-closed and $d^c$-closed. Furthermore, there are no $d d^c$-exact $1$-forms for degree reasons. Split $\alpha$ according to the bidegree as
    \[
    \alpha = \alpha^{1,0} + \alpha^{0,1}.
    \]
    Since $\alpha$ is both $d$-closed and $d^c$-closed, the forms $\alpha^{1,0}$ and $\alpha^{0,1}$ are both $d$-closed (thus $d^c$-closed by remark \ref{rem:bigraded}), so that they define two Bott-Chern cohomology classes in $H^{1,0}_{d + d^c}$, resp.\ $H^{1,0}_{d + d^c}$. The converse inclusion is immediate.
\end{proof}

\begin{remark}
The decomposition of Bott-Chern cohomology of $1$-forms can also be seen as a consequence of a more general decomposition (Corollary \ref{cor:BC:splitting}). A similar splitting for $2$-forms cannot hold (Example \ref{ex:no:decomposition}).
\end{remark}

\subsection{Cohomologies of the operators \texorpdfstring{$\delta$}{} and \texorpdfstring{$\bar \delta$}{}.}\label{sec:delta:cohom}

In general, the operators $\delta$ and $\bar \delta$ anti-commute but do not square to zero, and their cohomology is not well-defined. Note that, by \eqref{eq:delta2}, forms where $d d^c + d^c d = 0$ coincide with forms where $\delta^2=0$. Thus the subcomplex $B^\bullet$ appears to be the natural space on which to define the cohomologies of $\delta$, $\bar \delta$.

\begin{definition}\label{def:delta:cohom}
    Let $(M,J)$ be an almost complex manifold. The \emph{$\delta$-cohomology of $(M,J)$} is
    \[
    H^k_\delta (M,J) := \frac{\ker (\delta \colon B^k \longrightarrow B^{k+1})}{\Ima (\delta \colon B^{k-1} \longrightarrow B^k)}.
    \]
    The \emph{$\bar \delta$-cohomology of $(M,J)$} is 
    \[
    H^k_{\bar \delta} (M,J) := \frac{\ker (\bar \delta \colon B^k \longrightarrow B^{k+1})}{\Ima (\bar \delta \colon B^{k-1} \longrightarrow B^k)}.
    \]
\end{definition}

Following definition \ref{def:BCA:cohomology}, one can define Bott-Chern and Aeppli cohomologies built using $\delta$, $\bar \delta$, instead of $d$, $d^c$.

\begin{definition}\label{def:BCA:delta}

Let $(M,J)$ be an almost complex manifold. The \emph{$(\delta + \bar \delta)$-cohomology} of $(M,J)$ is
    \[
    H^k_{\delta + \bar \delta} (M,J) := \frac{\ker (\delta \colon B^k \rightarrow B^{k+1}) \cap \ker (\bar \delta \colon B^k \rightarrow B^{k+1})}{ \Ima (\delta \bar \delta \colon B^{k-2} \rightarrow B^k)}.
    \]
    The \emph{$\delta \bar \delta$-cohomology of $(M,J)$} is 
    \[
    H^k_{dd^c} (M,J) := \frac{\ker (\delta \bar \delta \colon C^k \rightarrow C^{k+2}) }{ \Ima (\delta \colon C^{k-1} \rightarrow C^k) + \Ima (\bar \delta \colon C^{k-1} \rightarrow C^k)}.
    \]
\end{definition}

In the integrable case, Bott-Chern and Aeppli cohomologies built using the operators $d$, $d^c$ and $\partial$, $\bar \partial$ coincide. This is still true in the non-integrable case for Bott-Chern and Aeppli cohomologies defined using the operators $d$, $d^c$ and $\delta$, $\bar \delta$.

\begin{proposition}\label{prop:BCA:equivalence}
    Let $(M,J)$ be an almost complex manifold. Then
    \begin{equation}
        H^\bullet_{d + d^c} (M,J) = H^\bullet_{\delta + \bar \delta} (M,J) \quad  \text{and} \quad H^\bullet_{d d^c} (M,J) = H^\bullet_{\delta \bar \delta} (M,J).
    \end{equation}
\end{proposition}

\begin{proof}
    By \eqref{eq:delta:ddc} and \eqref{eq:ddc:delta}, we have that
    \[
    \ker d \cap \ker d^c = \ker \delta \cap \ker \bar \delta.
    \]
    By \eqref{eq:deltabardelta}, one has
    \[
    \Ima (d d^c \colon B^{k-2} \longrightarrow B^k) =  \Ima (\delta \bar \delta \colon B^{k-2} \longrightarrow B^k),
    \]
    so that we conclude that
    \[
    H^k_{d + d^c} = H^k_{\delta + \bar \delta}.
    \]
    The equality between Aeppli cohomology and $\delta \bar \delta$-cohomology follows with a similar reasoning, since we are computing the cohomologies on forms in $C^\bullet$.
\end{proof}

\subsection{\texorpdfstring{$\Z_2$}{}-decomposition of Bott-Chern cohomology.}\label{sec:even:odd}

The use of the operators $d$, $d^c$ (or, equivalently, of $\delta$, $\bar \delta$) to define Bott-Chern cohomology naturally induces a $\Z_2$-splitting of the cohomology. In this section we explain how to define such a splitting and how it is related to a generalization of the $J$-invariant and $J$-anti-invariant cohomologies introduced by Draghici, Li and Zhang.
\vspace{.2cm}

First, we observe that the bigrading induces a splitting of a $k$-form into an even and an odd part. We set
\begin{equation}\label{eq:splitting:parity}
    A^k_{ev} := \bigoplus_{\substack{p+q =k \\ p \text{ even}}} A^{p,q}, \quad \quad A^k_{od} := \bigoplus_{\substack{p+q =k \\ p \text{ odd}}} A^{p,q}.
\end{equation}

\begin{definition}\label{def:even:odd:forms}
If $\alpha \in A^k_{ev}$ , we say that $\alpha$ is an \emph{even $k$-form}. Similarly, if $\alpha \in A^k_{od}$ , we say that $\alpha$ is an \emph{odd $k$-form}. 
\end{definition}

This provides a direct sum
\begin{equation}\label{eq:even:odd}
A^k = A^k_{ev} \oplus A^k_{od},
\end{equation}
that allows to split $\alpha \in A^k$ into an \emph{even part} and an \emph{odd part} 
\[
\alpha = \alpha^{ev} + \alpha ^{od},
\]
given by the projections on the subspaces defined in \eqref{eq:splitting:parity}. The same splitting can be given for forms in $B^k$. It is immediate from the definition to check that the wedge product of forms of the same parity is even, while the wedge product of forms of opposite parity is odd. Thus the decomposition \eqref{eq:even:odd} endows $A^\bullet$ with a $\Z_2$-graded algebra structure. With respect to the $\Z_2$-grading on $A^k$, we have that complex conjugation sends even forms into odd forms, and vice versa.

\begin{definition}\label{def:even:odd:operator}
Let $P \colon A^\bullet \rightarrow A^\bullet$ be a differential operator. We say that $P$ is \emph{even} if it preserves the parity of forms, i.e., if
\[
P( A^\bullet_{ev}) \subseteq A^\bullet_{ev} \quad \text{and} \quad P( A^\bullet_{od}) \subseteq A^\bullet_{od}.
\]
We say that $P$ is \emph{odd} if it changes the parity of forms, i.e., if
\[
P( A^\bullet_{ev}) \subseteq A^\bullet_{od} \quad \text{and} \quad P( A^\bullet_{od}) \subseteq A^\bullet_{ev}.
\]
\end{definition}
The differential $d$ is neither even nor odd, but it admits a splitting into an even and an odd part, that correspond to the operators $\delta$, $\bar \delta$. Indeed, when computing $d$ on even forms, $\delta$ is the projection of $d$ onto odd forms, while $\bar \delta$ is the projection onto even forms. The parities of most of the operators we study can be determined without effort and are collected in the next proposition.

\begin{proposition}\label{prop:parity}
    We have the following properties:
    \begin{itemize}
        \item [$(i)$] the composition of operators of the same parity is even;

        \item [$(ii)$] the composition of operators of opposite parity is odd;
        
        \item [$(iii)$] the operator $\delta$ is odd while $\bar \delta$ is even;

        \item [$(iv)$] the operator $ 4 i \delta \bar \delta = d d^c - d^c d$ is odd; 

         \item [$(v)$] the operator $ - 4 i \delta^2 = d d^c + d^c d$ is even;

         \item [$(vi)$] the operator $d d^c$ restricted to $B^\bullet$ is odd.
         
    \end{itemize}
\end{proposition}

\begin{proof}
    We have that $(i)$ and $(ii)$ follow immediately from the definition of even and odd operators, $(iii)$ follows from the explicit expression of $\delta$ and $\bar \delta$ in terms of $\mu$, $\partial$, $\bar \partial$, $\bar \mu$, while $(iv)$ and $(v)$ are a consequence of $(i)-(iii)$. For $(vi)$, let $\alpha \in B^\bullet$ be an even form. Then $(d d^c + d^c d)\alpha = - 4 i \delta^2 \alpha = 4 i \bar \delta^2 \alpha = 0$. In particular, by \eqref{eq:ddc:delta}, we also have that
    \[
    d d^c \alpha = i (\delta + \bar \delta)(\bar \delta - \delta) \alpha = i (\bar \delta^2 + 2 \delta \bar \delta -  \delta^2) \alpha = 2 i \delta \bar \delta \alpha.
    \]
    Thus on $B^\bullet$, the operator $d d^c$ coincides with $\delta \bar \delta$ up to a constant and it is odd.
\end{proof}

Using even and odd forms, one can define an associated de Rham cohomology.

\begin{definition}\label{def:even:odd:cohomology}
    Let $(M,J)$ be an almost complex manifold. The \emph{$J$-even cohomology} of $(M,J)$ is 
    \[
    H^k_{ev} (M,J) := \{ [\alpha] \in H^k_{dR}: \alpha \in A^k_{ev} \}.
    \]
    The \emph{$J$-odd cohomology} of $(M,J)$ is 
    \[
    H^k_{od} (M,J) := \{ [\alpha] \in H^k_{dR}: \alpha \in A^k_{od} \}.
    \]
\end{definition}
The $J$-even and $J$-odd cohomologies are a special case of the cohomology $H^S_J$ \cite{AT11} obtained by taking $S = \{ (p,q) : \text{$p$ is even} \}$, resp.\ $S = \{ (p,q) : \text{$p$ is odd} \}$. They also generalize the notions of cohomologies given by Draghici, Li and Zhang. Indeed, if $k=2$, we have that
\[
H^2_{ev} = H^{(2,0)(0,2)}_J \quad \text{and} \quad  H^2_{od} = H^{(1,1)}_J,
\]
where $H^{(2,0)(0,2)}_J$ and $H^{(1,1)}_J$ are generalizations of Dolbeault cohomology groups to almost complex manifolds \cite{LZ09}. Moreover, if $M$ is a $4$-manifold, then
\[
H^2_{ev} = H^-_J \quad \text{and} \quad H^2_{od} = H^+_J,
\]
where $H^-_J$, resp.\ $H^+_J$, is the $J$-anti-invariant cohomology, resp.\ $J$-invariant cohomology, of $(M,J)$ \cite{DLZ10}. Apart from generalizing well-known cohomologies, $J$-even and $J$-odd cohomologies arise naturally in the context of Bott-Chern cohomology and provide a splitting of $H^k_{d + d^c}$.

\begin{theorem}\label{thm:BC:splitting}
    There is a natural map
    \begin{align}\label{eq:splitting}
    \begin{split}
    H^k_{d + d^c} &\longrightarrow H^k_{ev} + H^k_{od}, \\
    [\alpha]_{d + d^c} &\longmapsto [\alpha^{ev}]_d + [\alpha^{od}]_d.
    \end{split}
    \end{align}
\end{theorem}

\begin{proof}
    The theorem is a consequence of the fact that a $k$-form $\alpha$ is $d$-closed and $d^c$-closed if and only if its even and odd part are both $d$-closed (cf.\ remark \ref{rem:bigraded}). Decompose $\alpha = \alpha^{ev} + \alpha^{od}$. The almost complex structure $J$ acts on even and odd forms as multiplication by a constant, and the two constants differ by a sign. More precisely, we have that
    \[
    J \alpha^{ev} = (-i)^k \alpha^{ev} \quad \text{and} \quad J \alpha^{od} = - (-i)^k \alpha^{od},
    \]
    so that 
    \begin{align*}
    d \alpha = 0 &\Leftrightarrow d \alpha^{ev} + d \alpha^{od} =0, \quad \text{and} \\
    d^c \alpha = 0 &\Leftrightarrow d \alpha^{ev} - d \alpha^{od} =0.
    \end{align*}
    In particular, $\alpha$ defines a Bott-Chern cohomology class if and only if $\alpha^{ev}$ and $\alpha^{od}$ define de Rham cohomology classes, giving the map \eqref{eq:splitting}.
\end{proof}

We recall that, if $k=2$ and $(M,J)$ is a compact almost complex $4$-manifold, then de Rham cohomology splits as the direct sum of $J$-invariant and $J$-anti-invariant cohomologies \cite[Theorem 2.3]{DLZ10}. However, since de Rham cohomology classes are defined up to $d$-exact forms and $d$ is neither an even nor odd operator, we have that $H^k_{ev} \cap H^k_{od} \neq \{ 0 \}$. In particular, in general there is no splitting 
\[
H^k_d = H^k_{ev} \oplus H^k_{od}.
\]
However, we can define cohomologies in the context of Bott-Chern cohomology that take into account the parity of the forms involved.

\begin{definition}\label{def:even:odd:BC}
Let $(H^k_{d+d^c})^{ev}$ be the space of Bott-Chern cohomology classes computed on even forms, i.e.,
\[
(H^k_{d+d^c})^{ev} = \frac{\ker d \cap \ker d^c \cap A^k_{ev}}{\Ima ( dd^c \colon B^{k-2}_{od} \rightarrow B^{k}_{ev})},
\]
and let $(H^k_{d+d^c})^{od}$ be the space of Bott-Chern cohomology classes computed on odd forms, i.e.,
\[
(H^k_{d+d^c})^{od} = \frac{\ker d \cap \ker d^c \cap A^k_{od}}{\Ima ( dd^c \colon B^{k-2}_{ev} \rightarrow B^{k}_{od})}.
\]
\end{definition}
Using these cohomologies, we can prove the desired splitting.

\begin{cor}\label{cor:BC:splitting}
    The map \eqref{eq:splitting} induces a decomposition of Bott-Chern cohomology
    \begin{equation}\label{eq:even:odd:decomposition}
        H^k_{d+d^c} = (H^k_{d+d^c})^{ev} \oplus (H^k_{d+d^c})^{od}.
    \end{equation}
\end{cor}

\begin{proof}
    Let $[\alpha]_{d + d^c} \in H^k_{d+d^c}$ be a Bott-Chern cohomology class. Writing the representatives according to the parity of the forms we have that 
    \[
    [\alpha]_{d + d^c} = \alpha^{ev} + \alpha^{od} + (dd^c \beta)^{ev} + (d d^c \beta)^{od},
    \]
    where $\beta \in B^{k-2}$. By proposition \ref{prop:parity} $(vi)$, we know that $dd^c$ is an odd operator. Furthermore $B^{k-2}$ also splits into even and odd part, so that 
    \[
    (dd^c \beta)^{ev} = d d^c (\beta^{od}) \quad \text{and} \quad (dd^c \beta)^{od} = d d^c (\beta^{ev}),
    \]
    proving that the even and odd parts of a $dd^c$-exact form are both $d d^c$-exact. Moreover, since $d d^c + d^c d$ is even, the even and odd part of a form in $B^\bullet$ are both still in $B^\bullet$. This implies that the even and odd part of any representative provide a well-defined splitting of Bott-Chern cohomology.
\end{proof}

\begin{remark}
    Proposition \ref{prop:decomposition:1forms} follows as a corollary of the splitting of Bott-Chern cohomology between even and odd forms in degree $1$.
\end{remark}

\begin{remark}
    The splitting between even and odd forms appears naturally in the context of the symplectic cohomologies, as it has been already observed by Tseng and Yau \cite[Section 5]{TY12a}
\end{remark}

\section{Hodge Theory on almost complex manifolds}\label{sec:hodge:theory}

In this section we study the properties of elliptic, self-adjoint operators built using the operators $d$, $d^c$. We describe the associated spaces of harmonic forms and we prove that there is an injection of such spaces into Bott-Chern and Aeppli cohomologies. We also investigate to which extent their dimension $h^k_{d+d^c}$ is metric independent. As a consequence, we are able to use $h^1_{d+d^c}$ to distinguish between different almost complex structures. Finally, we relate our spaces of harmonic forms with harmonic forms built using $\delta$, $\bar \delta$.

\subsection{Bott-Chern and Aeppli Laplacians on almost Hermitian manifolds.}\label{sec:BCA:laplacians}

Let $(M,J,g)$ be a \textbf{compact} almost Hermitian manifold. In general, we have that $d$ and $ d^c$ do not anti-commute and there are four different Laplacians that generalize Bott-Chern and Aeppli Laplacians.

\begin{definition}\label{def:laplacians}
    The \emph{$(d + d^c)$-Laplacian} is 
    \begin{equation}\label{def:d+dc:laplacian}
    \Delta_{d + d^c} := d d^c ( d d^c)^* + ( d d^c)^* d d^c + d^* d^c ( d^* d^c )^* + ( d^* d^c )^* d^* d^c + d^* d + (d^c)^* d^c.
    \end{equation}

    The \emph{$(d^c + d)$-Laplacian} is 
    \begin{equation}\label{def:dc+d:laplacian}
    \Delta_{d^c + d} := d^c d ( d^c d)^* + ( d^c d)^* d^c d + (d^c)^* d ( (d^c)^* d )^* + ( (d^c)^* d )^* (d^c)^* d + d^* d + (d^c)^* d^c.
    \end{equation}

    The \emph{$d d^c$-Laplacian} is 
    \begin{equation}\label{def:ddc:laplacian}
    \Delta_{d d^c} = d d^c ( d d^c)^* + ( d d^c)^* d d^c +  d (d^c)^* ( d (d^c)^* )^* + ( d (d^c)^* )^* d (d^c)^* + d d^* + d^c (d^c)^* .
    \end{equation}
    
    The \emph{$d^c d$-Laplacian} is 
    \begin{equation}\label{def:dcd:laplacian}
    \Delta_{d^c d} = d^c d ( d^c d)^* + ( d^c d)^* d^c d + d^c d^* ( d^c d^* )^* + ( d^c d^* )^* d^c d^* + d d^* + d^c (d^c)^*.
    \end{equation}
    
\end{definition}
The operators $\Delta_{d + d^c}$ and $\Delta_{d^c + d}$ provide a generalization to the non-integrable case of $\Delta_{BC}$, while $\Delta_{d d^c}$ and $\Delta_{d^c d} $ are a generalization of $\Delta_A$. Indeed, if $d$ and $d^c$ anti-commute then $\Delta_{BC} = \Delta_{d + d^c} = \Delta_{d^c + d}$ and $\Delta_A = \Delta_{d d^c} = \Delta_{d^c d}$. \\
Note that, as in the integrable case, we have that
\[
* \Delta_{d+d^c} = \Delta_{d d^c} * \quad \text{and} \quad * \Delta_{d^c + d} = \Delta_{d^c d}*.
\]
Moreover, we also have
\[
\Delta_{d+d^c} J = J \Delta_{d^c + d} \quad \text{and} \quad \Delta_{d d^c } J = J \Delta_{d^c d}.
\]
We study the kernels of such operators.

\begin{definition}\label{def:P:harmonic}
    Let $P \in \{ d+ d^c, d^c +d, d d^c, d^c d \}$. A $k$-form $\alpha \in A^k$ is said to be \emph{$P$-harmonic} if $\Delta_P (\alpha) =0$.\\
    We denote the space of $P$-harmonic $k$-forms by $\H^k_P$.
\end{definition}

If $\alpha$ is $P$-harmonic, using the equation $\langle \Delta_P (\alpha), \alpha \rangle =0 $, one can explicitly write the spaces of $P$-harmonic forms. More precisely, we have that
\begin{align*}
    &\H^k_{d+d^c} = \{ \alpha \in A^k : d \alpha =0, \, d^c \alpha =0, \, (d d^c)^* \alpha =0 \}, \\
    &\H^k_{d^c+d} = \{ \alpha \in A^k : d \alpha =0, \, d^c \alpha =0, \, (d^c d)^* \alpha =0 \}, \\
    &\H^k_{d d^c} = \{ \alpha \in A^k : d^* \alpha =0, \, (d^c)^* \alpha =0, \, d d^c \alpha =0 \}, \\
    &\H^k_{d^c d} = \{ \alpha \in A^k : d^* \alpha =0, \, (d^c)^* \alpha =0, \, d^c d \alpha =0 \}.
\end{align*}

Regarding the ellipticity of Bott-Chern and Aeppli Laplacians, we can prove the following
\begin{proposition}\label{prop:ellipticity}
    Let $P \in \{ d+ d^c, d^c +d, d d^c, d^c d \}$. Then $\Delta_P$ is a $4^\text{th}$-order, self-adjoint, elliptic operator. There is a decomposition
    \[
    A^k = \H^k_P \overset{\perp}{\oplus} \Ima \Delta_P.
    \]
    Furthermore, $\H^k_P$ is a finite dimensional vector space over $\C$.
\end{proposition}

\begin{proof}
    The fact that $\Delta_P$ is a $4^\text{th}$-order and self-adjoint operator follows from the definition. We prove it is elliptic by proving that it coincides, up to terms of lower order, with an elliptic operator. We write down the proof for $P = d + d^c$. The remaining cases follow similarly. Denote by $\cong$ the equality up to lower order terms. We have that:
    \begin{itemize}
        \item $d d^c + d^c d$ has order one, so that $d d^c \cong - d^c d$;

        \item $d (d^c)^* + (d^c)^* d$ has order one by the K\"ahler identities for almost Hermitian manifolds \cite{CW20a, FH22}, so that $d (d^c)^* \cong - (d^c)^* d$;

        \item $\Delta_d$ and $\Delta_{d^c}$ are elliptic.
    \end{itemize}
    Thus, we can conclude that 
    \begin{align*}
        \Delta_{d + d^c} &\cong d d^c ( d d^c)^* + ( d d^c)^* d d^c + d^* d^c ( d^* d^c )^* + ( d^* d^c )^* d^* d^c \cong\\
        &\cong d d^* d^c (d^c)^* + d^* d ( d^c)^* d^c + d^* d d^c (d^c )^* + d d^* ( d^c )^* d^c = \\
        &= \Delta_d \Delta_{d^c} \cong (\Delta_d)^2,
    \end{align*}
    that is elliptic. The orthogonal direct sum decomposition between image and kernel of $\Delta_P$ and the finite dimensionality of the kernel follow from the theory of self-adjoint, elliptic operator on compact manifolds.
\end{proof}

We set
\begin{equation}\label{bc:numbers}
    h^k_P := \dim_\C \H^k_P.
\end{equation}
In principle, $h^k_P$ depends on the choice of $J$ and of $g$. When it is not clear from the context to which almost Hermitian structure we are referring to, we will underline such a dependence by writing $h^k_P(J,g)$.\\
The numbers $h^k_P$ define invariants of almost Hermitian manifolds, as we prove in the following proposition.

\begin{theorem}\label{thm:ah:invariants}
    Let $(M,J,g)$ and $(M',J', g')$ be two almost Hermitian manifolds of the same dimension and let $f \colon M \rightarrow M'$ be a surjective pseudo-holomorphic map such that $f^* g' = g$. Then
    \[
    h^k_P (J',g') \le h^k_P (J, g).
    \]
    Moreover, if $f$ is also a diffeomorphism then
    \[
    h^k_P (J',g') = h^k_P (J, g).
    \]
\end{theorem}

\begin{proof}
    If $\alpha$ is $P$-harmonic with respect to the almost Hermitian structure $(J',g')$, then $f^* \alpha$ is $P$-harmonic with respect to $(J, g)$. Furthermore, the differential $df$ is injective because $f^*g'=g$. Since $M$ and $M'$ have he same dimension, $df$ is also surjective at every point of $M$. We show that the pull-back $f^*$ is injective. Indeed, let $\alpha\in A^k (M')$ and assume that for any $x\in M$ we have $f^*\alpha|_x=0$. Let $x'$ be any point in $M'$. Let $v_1',\ldots , v_k'\in T_{x'}M'$. Since the maps $f$ and $df$ are surjective, then there exist $v_1,\ldots,v_k\in T_xM$ such that $df|_x v_j=v_j'$, where $j=1,\ldots , k$, for a suitable $x\in M$. Therefore,
    \[
    \alpha|_{x'}(v_1',\ldots , v_k')=\alpha|_{x'}(df|_xv_1,\ldots , df|_xv_k)=f^*\alpha|_x(v_1,\ldots , v_k)=0.
    \]
    Hence, for any given $x'\in M'$, we have that $\alpha|_{x'}=0$ and consequently $\alpha=0$. This proves the inequality $h^k_P (J',g') \le h^k_P (J, g)$.

    If $f$ is also a diffeomorphism, by applying the above argument to 
    \[
    f^{-1}:M'\to M,
    \]
    we derive that $h^k_P(J,g)\leq h^k_{P'}(J',g')$. Therefore,
    \[
    h^k_P(J,g)= h^k_{P'}(J',g').
    \]
    
\end{proof}

In the next proposition we collect the relations occurring among the spaces of $P$-harmonic forms.

\begin{proposition}\label{prop:harmonic:iso}
There is a commutative diagram
        
        \begin{center}
        \begin{tikzcd}
        \mathcal{H}^k_{d + d^c} \arrow[rr, "\sim", "J"'] \arrow[dd, "\wr", "*"'] &  & \mathcal{H}^k_{d^c + d} \arrow[dd, "\wr", "*"'] \\
                                              &  &                                    \\
\mathcal{H}^{2m-k}_{d d^c} \arrow[rr, "\sim", "J"']         &  & \mathcal{H}^{2m-k}_{d^c d}        
        \end{tikzcd}
        \end{center}
\end{proposition}

\begin{proof}
    Since $g$ is $J$-compatible, the Hodge $*$ operator and $J$ commute. The isomorphism between $(d+d^c)$-harmonic forms and $(d^c+d)$-harmonic forms follows by noting that if $\alpha$ is $d$-closed, $d^c$-closed and $(d d^c)^*$-closed, then $J \alpha $ is $d$-closed, $d^c$-closed and $(d^c d)^*$-closed. The other isomorphisms in the diagram follow with a similar reasoning.
\end{proof}

Thanks to proposition \ref{prop:harmonic:iso}, the study of the invariants $h^k_P$ reduces to computing the dimensions of the spaces of $(d+d^c)$-harmonic forms.

\begin{cor}
Let $(M,J,g)$ be an almost Hermitian manifold. Then
\[
h^k_{d+d^c} = h^k_{d^c + d} = h^{2m-k}_{d d^c} = h^{2m-k}_{d^c d}.
\]
\end{cor}

\subsection{The numbers \texorpdfstring{$h^k_{d+d^c}$}{} as an almost complex invariant.}\label{sec:ac:hodge:numbers}

We study to which extent the numbers $h^k_{d+d^c}$ do not depend on the choice of the compatible metric $g$, and thus they define an invariant of the almost complex structure $J$.
\vspace{.2cm}

We begin with the following result:

\begin{proposition}\label{prop:metric:indep}
    Let $(M,J)$ be an almost complex manifold. For every choice of compatible metric, we have that $h^0_{d+d^c} = 1$, $h^{2m}_{d+d^c} = 1$, and that $h^1_{d+d^c}$ is metric independent.
\end{proposition}
\begin{proof}
        In degree $k \in \{ 0,1,2m \}$, we can describe the space of $(d+d^c)$-harmonic forms. The space $\H^0_{d+d^c}$ consists only of constant functions, while $\H^{2m}_{d+d^c}$ consists of multiples of the volume form by a constant. This implies that $h^0_{d+d^c} = 1$ and $h^{2m}_{d+d^c} = 1$. The space of $(d+d^c)$-harmonic $1$-forms is given by
    \[
    \H^1_{d+d^c} = \H^1_{d^c +d} = \{ \alpha \in A^1 : d \alpha =0, \, d^c \alpha =0 \}.
    \]
    Since the conditions $d \alpha =0$ and $d^c \alpha =0$ do not depend on the choice of the metric, it follows that $h^1_{d+d^c}$ is metric independent.
\end{proof}

Even though in general the numbers $h^k_{d+d^c}$ will depend on the metric, $h^1_{d +d^c}$ is a non-trivial almost complex invariant that often allows to distinguish between almost complex structures. This happens even when other invariants fail in doing so.

\begin{proposition}\label{prop:h1:same:rank}
    There exists a smooth manifold $M$ admitting two almost complex structures $J_1$ and $J_2$ such that
    \[
    h^1_{d+d^c} (J_1) \neq  h^1_{d+d^c} (J_2).
    \]
    In particular, $h^1_{d+d^c}$ allows to distinguish between almost complex structures. Moreover, one can take $J_1$ and $J_2$ in such a way that
    \[
    \rk N_{J_1 |_x} = \rk N_{J_2 |_x}
    \]
    at every point $x \in M$ and $h^1_{d+d^c} (J_1) \neq  h^1_{d+d^c} (J_2)$. In particular, $h^1_{d+d^c}$ allows to distinguish between almost complex structures whose Nijenhuis tensors have the same rank.
\end{proposition}

\begin{proof}
    Let $M$ be the holomorphically parallelizable Nakamura manifold \cite{Nak75}, with basis of holomorphic $(1,0)$-forms
    \[
    \phi^1 = dz^1, \quad \phi^2 = e^{-z_1} dz^2, \quad \phi^2 = e^{z_1} dz^3,
    \]
    and differentials
    \[
    d \phi^1 =0, \quad d\phi^2 = - \phi^{12}, \quad d \phi^3 = \phi^{13}.
    \]
    Let $J_1$ be the almost complex structure defined by the $(1,0)$-forms
    \[
    \omega^1 = \phi^1, \quad \omega^2 = \phi^2 + \phi^{\bar 3}, \quad \omega^3 = \phi^3.
    \]
    Then $\rk N_{J_1} = 1$ at every point \cite{ST22}, and it is immediate to check that 
    \[
    h^1_{d + d^c} (J_1) = 2 \dim_\C ( \ker d \cap A^{1,0}) = 2.
    \]
    Let $J_2$ be the almost complex structure defined by the $(1,0)$-forms
    \[
    \omega^1 = f \phi^1, \quad \omega^2 = \phi^2 + \phi^{\bar 3}, \quad \omega^3 = \phi^3,
    \]
    with $f \in C^\infty(M)$ never-vanishing and non-constant. Then we still have $\rk N_{J_2} = 1$, but $h^1_{d + d^c} (J_2) =0$, proving the proposition. Examples in higher dimension can be found taking products of $M$ with a complex torus.
\end{proof}

\begin{remark}
    More examples where $h^1_{d+d^c}$ distinguishes between almost complex structures can be found on the torus $T^{2m}$. There exists almost complex structures $J_k$ on $T^{2m}$, with $k \in \{ 0, \dots, m \}$ and $2m \ge 6$, such that $\rk N_{J_k} = k$ \cite{CPS22}. A straightforward computation shows that \emph{for those specific structures}, $h^1_{d + d^c} = 2(m -k)$.
\end{remark}

Proposition \ref{prop:h1:same:rank} shows that the number $h^1_{d + d^c}$ distinguishes between almost complex structures whose Nijenhuis tensors have both rank 1. However, it cannot distinguish between maximally non-integrable almost complex structures on manifolds of dimension at least $6$.

\begin{proposition}\label{prop:h1:mni}
    Let $(M,J)$ be an almost complex $2m$-manifold. If $2m \ge 6$ and there exists $x\in M$ such that $\rk N_J|_x$ is maximum, then $h^1_{d + d^c} =0$.
\end{proposition}

\begin{proof}
    By assumption, the map $\bar \mu_x \colon A^{1,0}_x \rightarrow A^{0,2}_x$ is injective. Therefore, there are no global $d$-closed $(1,0)$-forms. Indeed, assume by contradiction that $\alpha \in A^{1,0}$ is $d$-closed and non-zero. Then in a neighborhood of $x \in M$ we would have $\bar \mu \alpha =0$, contradicting the assumption that $\bar \mu_x$ is injective.
\end{proof}

A similar statement for $4$-manifolds cannot hold. For instance, the almost complex structure on the Kodaira-Thurston we consider in example \ref{ex:KT} is maximally non-integrable but $h^1_{d + d^c} = 2$.

\subsection{Relation with cohomology.}\label{sec:cohom:iso}

A natural question to ask is whether or not Bott-Chern and Aeppli cohomologies introduced in section \ref{sec:BCA:cohom} are isomorphic to some space of harmonic forms. In general, this is not the case since proposition \ref{prop:KT:cohomology} shows that $H^2_{d+d^c}$ might be infinite dimensional even on compact manifolds. However, we have an equality between Bott-Chern cohomology and $(d+d^c)$-harmonic forms in degree $k \in \{ 0,1 \}$ and an inclusion in other degrees. A similar result holds for Aeppli cohomology, providing an isomorphism $H^1_{d  + d^c} \cong H^{2m-1}_{d d^c}$.

\begin{proposition}\label{prop:BC:harmonic:inclusion}
    Let $(M,J)$ be an almost complex manifold. Then
    \[
    H^0_{d+d^c} = \H^0_{d+d^c} = \C \quad \text{and} \quad H^1_{d+d^c} = \H^1_{d+d^c}.
    \]
    In particular, $H^1_{d+d^c}$ is finite dimensional. Furthermore, there is an inclusion
    \[
    \H^k_{d+d^c} \cup \H^k_{d^c+d} \longhookrightarrow H^k_{d+d^c}.
    \]
\end{proposition}

\begin{proof}
    The fact that $H^0_{d+d^c} = \H^0_{d+d^c} = \C$ and  $H^1_{d+d^c} = \H^1_{d+d^c}$ follows immediately from the explicit expression of Bott-Chern cohomology and of $(d+d^c)$-harmonic forms (cf.\ proposition \ref{prop:metric:indep}). Let now $\alpha \in \H^k_{d + d^c}$. Then $d \alpha =0$ and $d^c \alpha =0$, so that the identity defines a map from $\H^k_{d+d^c}$ to $H^k_{d+d^c}$. Assume that $\alpha$ defines the zero class in Bott-Chern cohomology, i.e., that $\alpha = d d^c \beta$, with $\beta \in B^\bullet$. Since $\alpha$ is harmonic, we have that
    \[
    0 = (d d^c)^* \alpha = (d d^c)^* d d^c \beta, 
    \]
    and $0 = \langle (d d^c)^* d d^c \beta, \beta \rangle = \lVert d d^c \beta \rVert^2$. This implies that $\alpha = d d^c \beta =0$, proving injectivity. The same argument applies to $(d^c + d)$-harmonic forms.
\end{proof}

\begin{remark}\label{rem:h1:even}
    From the splitting \eqref{BC:cohom:decomposition:h1} and using complex conjugation, we deduce that $h^1_{d + d^c}$ must be even. This is not the case for other degrees. For example, there are almost Hermitian $4$-manifolds such that $h^3_{d + d^c} = 3$ (example \ref{ex:KT}).
\end{remark}

A similar statement is true for Aeppli cohomology.

\begin{proposition}\label{prop:A:harmonic:inclusion}
    Let $(M,J)$ be an almost complex manifold. Then
    \[
    H^0_{d d^c} = \H^0_{d d^c} = \C,  \quad H^{2m}_{d d^c} \cong \H^{2m}_{d d^c} \cong  \C \quad \text{and} \quad H^{2m-1}_{d d^c} \cong \H^{2m-1}_{d d^c}.
    \]
    In particular, $H^{2m-1}_{d d^c}$ is finite dimensional. Furthermore, there is an inclusion
    \[
    \H^k_{dd^c} \cup \H^k_{d^c d} \longhookrightarrow H^k_{d d^c}.
    \]
\end{proposition}

\begin{proof}
    We first prove the injectivity of 
    \[
    \H^k_{dd^c} \longhookrightarrow H^k_{d d^c}.
    \]
    Injectivity for $\H^k_{d^c d} \hookrightarrow H^k_{d d^c}$ follows with a similar proof. Let $\alpha \in \H^k_{d d^c}$. Then $d d^c \alpha =0$, so that the projection on $C^k$ defines a map from $\H^k_{d d^c}$ to $H^k_{d d^c}$. Assume that $\alpha$ defines the zero class in Aeppli cohomology, i.e., that $\alpha = d \beta + d^c \gamma$. Since $\alpha$ is harmonic, we have that
    \[
    0 = d^* \alpha = d^* d\beta + d^* d^c \gamma \quad \text{and} \quad 0 = (d^c)^* \alpha = (d^c)^* d\beta + (d^c)^* d^c \gamma.
    \]
    As a consequence we obtain that
    \[
    \lVert d \beta \rVert^2 + \langle d\beta, d^c \gamma \rangle =0 \quad \text{and} \quad \lVert d^c \gamma \rVert^2 + \langle d\beta, d^c \gamma \rangle =0.
    \]
    In particular, $\lVert d \beta + d^c \gamma \rVert^2 =0 $, proving injectivity. \\
    We now prove the isomorphisms in degree $k \in \{0, 2m-1, 2m \}$. The statement for functions follows immediately from the definition. Let $\alpha \in A^{2m}$. Using the harmonic decomposition for $\Delta_{d d^c}$ of proposition \ref{prop:ellipticity}, we have that 
    \begin{align}\label{eq:2m:decomposition}
    \begin{split}
    \alpha &= c \, Vol + \Delta_{d d^c} \beta = \\
    &= c \, Vol + dd^c (d d^c)^* \beta + d (d^c)^* d^c d^* \beta  \\
    &+ d^c d^* d (d^c)^* \beta + d d^* \beta + d^c (d^c)^* \beta = \\
    &= c \, Vol + d\gamma + d^c \eta,
    \end{split}
    \end{align}
    where $c$ is a constant, and $\gamma$ and $\eta$ are $(2m-1)$-forms. Note that a similar decomposition holds also for $(2m-1)$-forms since $( d d^c)^* d d^c \beta =0$ if $\beta \in A^{2m-1}$. Passing to Aeppli cohomology, we obtain the existence of harmonic representatives. Uniqueness follows from the injectivity of $\H^k_{dd^c} \longhookrightarrow H^k_{d d^c}$.
\end{proof}

Using the relations between harmonic forms (Proposition \ref{prop:harmonic:iso}), we deduce an isomorphism for Bott-Chern and Aeppli cohomologies.

\begin{cor}\label{cor:duality}
    There is an isomorphism
    \[
    H^1_{d  + d^c} \cong H^{2m-1}_{d d^c}.
    \]
\end{cor}

\subsection{Relation with \texorpdfstring{$\H^k_\delta$, $\H^k_{\bar \delta}$, $\H^k_{\delta + \bar \delta}$, $\H^k_{\delta \bar \delta}$}{}.}\label{sec:delta:harmonic}

In section \ref{sec:delta:cohom} we explained the natural relation appearing between the operators $d$, $d^c$ and $\delta$, $\bar \delta$, and between the corresponding Bott-Chern and Aeppli cohomologies (Proposition \ref{prop:BCA:equivalence}). A natural problem that arises in this context is to establish the relations occurring among the spaces of harmonic forms associated to such operators.
\vspace{.2cm}

The operators $\delta$, $\bar \delta$ and their spaces of harmonic forms have already been studied by Tardini and the second author \cite{TT20}. We begin by recalling the definition of such spaces, introducing a slightly different notation coherent with the one adopted in this paper.\\
The space of \emph{$\delta$-harmonic $k$-forms} is the space
\[
\H^k_\delta = \{ \alpha \in A^k: \delta \alpha =0, \, \delta^* \alpha =0 \}.
\]
The space of \emph{$\bar \delta$-harmonic $k$-forms} is the space
\[
\H^k_{\bar \delta} = \{ \alpha \in A^k: \bar \delta \alpha =0, \, \bar \delta^* \alpha =0 \}.
\]
The space of \emph{$(\delta + \bar \delta)$-harmonic $k$-forms} is the space
\[
\H^k_{\delta + \bar \delta} = \{ \alpha \in A^k: \delta \alpha =0, \, \bar \delta \alpha =0, \, (\delta \bar \delta)^* \alpha =0 \}.
\]
The space of \emph{$\delta \bar \delta$-harmonic $k$-forms} is the space
\[
\H^k_{\delta \bar \delta} = \{ \alpha \in A^k: \delta^* \alpha =0, \, \bar \delta^* \alpha =0, \delta \bar \delta \alpha =0 \}.
\]

\begin{remark}
    By proposition \ref{prop:BCA:equivalence}, there is no difference between defining the cohomologies using $d$, $d^c$ or $\delta$, $\bar \delta$. However, the corresponding spaces of harmonic forms are different. Indeed, on the Kodaira-Thurston manifold endowed with the almost complex structure of example \ref{ex:KT}, we have that $h^3_{d + d^c} = 3$, while $\dim_\C \H^3_{\delta + \bar \delta}= 2$ \cite[Proposition 6.10 and example 7.1]{TT20}.
\end{remark}

For the sake of completeness, we collect in the next proposition several relations occurring among the spaces of harmonic forms, whose proof is straightforward.

\begin{proposition}\label{prop:harmonic:relations}
    A $k$-form is both $d$-harmonic and $d^c$-harmonic if and only if it is both $\delta$-harmonic and $\bar \delta$-harmonic, i.e.,
    \[
    \H^k_d \cap \H^k_{d^c} = \H^k_\delta \cap \H^k_{\bar \delta}.
    \]
    We have the following inclusions of harmonic forms:
    \begin{align*}
    &\H^k_{d+d^c} \cap \H^k_{d^c+d} \longhookrightarrow \H^k_{\delta + \bar \delta}, \\
    &\H^k_{d+d^c} \cap \H^k_{\delta + \bar \delta} \longhookrightarrow \H^k_{d^c+d}, \\
    &\H^k_{\delta + \bar \delta} \cap \H^k_{d^c+d} \longhookrightarrow \H^k_{d+d^c}, \\
    &\H^k_{d d^c} \cap \H^k_{d^c d} \longhookrightarrow \H^k_{\delta   \bar \delta}, \\
    &\H^k_{d d^c} \cap \H^k_{\delta   \bar \delta} \longhookrightarrow \H^k_{d^c d}, \\
    &\H^k_{\delta   \bar \delta} \cap \H^k_{d^c d} \longhookrightarrow \H^k_{d d^c}.
    \end{align*}
\end{proposition}

\begin{remark}
    In general equality among the spaces of harmonic forms considered in proposition \ref{prop:harmonic:relations} does not hold. For example, consider the Kodaira-Thurston manifold endowed with the almost complex structure of example \ref{ex:KT}. By \cite[Proposition 6.10 and example 7.1]{TT20} and direct computations, we have that
    \[
    \H^3_{\delta + \bar \delta} = \H^3_{d^c + d} = \C \langle \phi^{12\bar 2}, \, \phi^{2 \bar 1 \bar 2} \rangle,
    \]
    while, by proposition \ref{prop:KT:harmonic}, we know that $\H^3_{d + d^c} = \C \langle \phi^{12\bar 2}, \, \phi^{2 \bar 1 \bar 2}, \, \phi^{12 \bar 1} - \phi^{1 \bar 1 \bar 2} \rangle$. This shows the strict inclusion
    \[
     \H^3_{\delta + \bar \delta} \cap \H^3_{d^c + d} \subsetneq \H^3_{d + d^c}. 
    \]
\end{remark}

\section{Almost K\"ahler \texorpdfstring{$4$}{}-manifolds}\label{sec:ak:4manifolds}

In this section we study the space of $(d+d^c)$-harmonic forms on almost K\"ahler $4$-manifolds. We prove a decomposition for $\H^2_{d + d^c}$ between $d$-harmonic, anti-self-dual forms and $d$-harmonic, $J$-anti invariant forms. As a consequence, we deduce that $h^2_{d+d^c}$ does not depend on the choice of the $J$-compatible almost K\"ahler metric. We also investigate the relation between the spaces of $(d + d^c)$-harmonic forms and of $(d + d^\Lambda)$-harmonic forms considered by Tseng and Yau \cite{TY12a}.
\vspace{.2cm}

Let $(M,J)$ be an almost complex $4$-manifold and let $g$ be a $J$-compatible metric such that $(M,J,g,\omega)$ is almost K\"ahler. From proposition \ref{prop:metric:indep}, we know that for any almost Hermitian manifold the numbers $h^k_{d + d^c}$ are almost complex invariants for $k \in \{0,1,4 \}$.
\vspace{.2cm}

We begin by studying $(d + d^c)$-harmonic $2$-forms. On an almost Hermitian $4$-manifold, $A^2$ admits two different splittings, one between $J$-invariant and $J$-anti-invariant forms, and one between self-dual and anti-self-dual forms, namely
\[
A^2 = A^+_J \oplus  A^-_J =  A^+_g \oplus  A^-_g.
\]
The two splittings are related to each other via the fundamental form $\omega$ \cite{DLZ10}. We exploit the interplay of such decompositions to establish a decomposition for $\H^k_{d + d^c}$ when $\omega$ is $d$-closed. \\
Denote by $\H^-_g$ the space of anti-self-dual, $d$-harmonic $2$-forms and by $\H^{(2,0)(0,2)}_J$ the space of $J$-anti-invariant, $d$-harmonic $2$-forms.

\begin{theorem}\label{thm:decomposition}
Let $(M,J,g,\omega)$ be an almost K\"ahler manifold. Then 
\[
\H^2_{d + d^c} = \C \langle \omega \rangle \oplus \H^-_g \oplus \H^{(2,0)(0,2)}_J.
\]
\end{theorem}
\begin{proof}
    We prove that $\C \langle \omega \rangle \oplus \H^-_g \oplus \H^{(2,0)(0,2)}_J \subseteq \H^2_{d + d^c}$. Since $\omega$ is $d$-closed and has bidegree $(1,1)$, it is also $d^c$-closed. By the compatibility condition, we have that $* \omega = \omega$, thus $(d d^c)^* \omega =0$ and $z \omega$ is $(d+d^c)$-harmonic for all $z \in \C$. Forms in $\H^-_g$ are clearly $d$-closed and $(d d^c)^*$-closed. They are also $d^c$-closed since they have necessarily bidegree $(1,1)$. Finally, any form in $\H^{(2,0)(0,2)}_J$ is $d$-closed and $(d d^c)^*$-closed, but it is also $d^c$-closed since it is self-dual.\\
    For the opposite inclusion, let $\alpha$ be a $(d + d^c)$-harmonic $2$-form. Write $\alpha$ as a sum of its bidegree components
    \[
    \alpha = \alpha^{2,0} +  \alpha^{1,1} +  \alpha^{0,2}.
    \]
    Since $\alpha$ is both $d$-closed and $d^c$-closed, we have that  $\alpha^{2,0} +  \alpha^{0,2}$ and $\alpha^{1,1}$ are both $d$-closed (cf.\ proof of theorem \ref{thm:BC:splitting}). By bidegree reasons, the form $\alpha^{2,0} +  \alpha^{0,2}$ is self-dual and so it is also $d$-harmonic, providing an element in $\H^{(2,0)(0,2)}_J$. Consider now the Lefschetz decomposition
    \[
    \alpha^{1,1} = f \omega + \gamma^{1,1},
    \]
    with $f \in C^\infty(M)$ and $\gamma^{1,1}$ primitive. In particular, the fact that $\gamma^{1,1}$ is primitive implies that it is also anti-self-dual. Since $\alpha$ is $(d+d^c)$-harmonic, we have that 
    \begin{align*}
        0 = d d^c * \alpha &= d J^{-1} d J (\alpha^{2,0} +  \alpha^{0,2} + f \omega - \gamma ^{1,1})=\\
        &= d J^{-1} d (-\alpha^{2,0} - \alpha^{0,2} + f \omega - \gamma ^{1,1}) = \\
        &= d J^{-1} d (-\alpha + 2 f \omega ) = \\
        &= 2 d J^{-1} d f \wedge \omega=\\
        &= 2 d d^c f \wedge \omega = \\
        &= -4 \partial \bar \partial f \wedge \omega,
    \end{align*}
    where the first equality follows from the fact that $*$ is an isomorphism, the fifth equality by the $d$-closedness of $\omega$ and the last one by the fact that $\omega$ has bidegree $(1,1)$. The operator $f \mapsto \partial \bar \partial f \wedge \omega$ is elliptic \cite[proof of theorem 4.3]{PT22}, so $f$ must be constant. Since $f$ is constant, we have that
    \[
    0 = d \alpha^{1,1} = d \gamma^{1,1}.
    \]
    The form $\gamma^{1,1}$ is anti-self-dual and $d$-closed, hence $d$-harmonic, giving an element of $\H^-_g$. This completes the proof.
\end{proof}

The decomposition we proved has several consequences.

\begin{cor}\label{cor:upper:bound}
    If $(M,J,g,\omega)$ is an almost K\"ahler $4$-manifold, then
    \[
    h^2_{d + d^c} = b^- + 1 + h^-_J.
    \]
    In particular, we have that $h^2_{d + d^c} \le b_2$.
\end{cor}
\begin{proof}
    By theorem \ref{thm:decomposition}, we have that
    \[
    h^2_{d + d^c} = b^- + 1 + h^-_J.
    \]
    On an almost K\"ahler $4$-manifold, there is an inequality $b^- +1 \le h^+_J$ \cite[Proposition 3.1]{DLZ10} and a direct sum decomposition between $J$-invariant and $J$-anti-invariant cohomology \cite[Theorem 2.3]{DLZ10}. We conclude that 
    \[
    h^2_{d + d^c} = b^- + 1 + h^-_J \le h^+_J + h^-_J = b_2.
    \]
\end{proof}

\begin{cor}
    Let $(M,J)$ be an almost complex $4$-manifold admitting an almost K\"ahler metric. Then the number $h^2_{d + d^c}$ does not depend on the choice of the almost K\"ahler metric.
\end{cor}

Note that even though $h^2_{d+d^c}$ is an invariant of $J$, $\H^2_{d+d^c}$ does not coincide with $H^2_{d+d^c}$. Indeed, there exist compact almost K\"ahler $4$-manifolds for which the latter is infinite dimensional (example \ref{ex:KT}).
\vspace{.2cm}

We now study the space of $(d+d^c)$-harmonic $3$-forms. It turns out that it coincides with the space of $(d + d^\Lambda)$-harmonic $3$-forms.

\begin{theorem}\label{thm:h3}
    Let $(M,J,g,\omega)$ be an almost K\"ahler $4$-manifold. Then
    \[
    \H^3_{d + d^c} =  \H^3_{d + d^\Lambda}.
    \]
\end{theorem}
\begin{proof}
    We recall that $d^\Lambda = (d^c)^*$ and that the explicit expression of the space of harmonic forms that are involved is
    \begin{align*}
        &\H^3_{d+d^c} = \{ \alpha \in A^3 : d \alpha =0, \, d^c \alpha =0, \, (d d^c)^* \alpha =0 \},\\
        &\H^3_{d+d^\Lambda} = \{ \alpha \in A^3 : d \alpha =0, \, d^\Lambda \alpha =0, \, (d d^\Lambda)^* \alpha =0 \}.
    \end{align*}
    We prove that $\H^3_{d + d^c} \subseteq  \H^3_{d + d^\Lambda}$. Let $\alpha \in \H^3_{d + d^c}$. We write $\alpha$ using the Lefschetz decomposition and the bidegree as
    \[
    \alpha = L (\gamma^{1,0} + \gamma^{0,1}).
    \]
    Since $\alpha$ is both $d$-closed and $d^c$-closed, we have that 
    \[
    L d\gamma^{1,0} =0 \quad \text{and} \quad L d\gamma^{0,1} =0.
    \]
    The equation $(d d^c)^* \alpha =0$ can be written in terms of $d^\Lambda$ as 
    \begin{align*}
        0 &= d^* d^\Lambda \alpha = d^* d^\Lambda L (\gamma^{1,0} + \gamma^{0,1}) = \\
        &= d^* (d\Lambda - \Lambda d ) L (\gamma^{1,0} + \gamma^{0,1}) = \\
        &= d^* d \Lambda L (\gamma^{1,0} + \gamma^{0,1})= \\
        &= d^* d (\gamma^{1,0} + \gamma^{0,1}).
    \end{align*}
    In particular, this implies that 
    \begin{equation}\label{d:gamma}
        d (\gamma^{1,0} + \gamma^{0,1}) =0.
    \end{equation}
    To prove that $\alpha$ is $(d + d^\Lambda)$-harmonic, we observe that $d\alpha=0$ and $( d d^\Lambda)^* \alpha =0 $ since $d^* d^c \alpha =0$. Finally, we have that
    \[
    d^\Lambda \alpha = (d\Lambda - \Lambda d) \alpha = d \Lambda L (\gamma^{1,0} + \gamma^{0,1}) = d (\gamma^{1,0} + \gamma^{0,1}) =0,
    \]
    by \eqref{d:gamma}, proving the inclusion $\H^3_{d + d^c} \subseteq  \H^3_{d + d^\Lambda}$. \\
    For the opposite inclusion, let $\alpha \in \H^3_{d + d^\Lambda}$ and let $\alpha = L (\gamma^{1,0} + \gamma^{0,1})$ be its primitive decomposition. Then we have that
    \begin{equation}\label{d:gamma:bis}
    0 = d^\Lambda \alpha = d (\gamma^{1,0} + \gamma^{0,1}).
    \end{equation}
    Reasoning as above, this implies that $(d d^c)^* \alpha=0$. Using $(d d^\Lambda)^*$-closedness of $\alpha$, we can compute
    \begin{align*}
        0 &= d^* d^c \alpha = d^* J^{-1} d J L (\gamma^{1,0} + \gamma^{0,1}) =\\
        &= i \, d^* J^{-1} d L (\gamma^{1,0} - \gamma^{0,1}) =\\
        &= i \, d^* J^{-1} L (d \gamma^{1,0} - d \gamma^{0,1}).
    \end{align*}
    By bidegree reason and since we are on a $4$-manifold, the only parts of the differential that are left after applying $L$ are those of bidegree $(1,1)$, so that we have
    \begin{align*}
        0 &= i \, d^* J^{-1} L (\bar \partial \gamma^{1,0} - \partial \gamma^{0,1}) = \\
        &= i \, d^*  L (\bar \partial \gamma^{1,0} - \partial \gamma^{0,1}) =\\
        &= i \, d^*  L (d \gamma^{1,0} - d \gamma^{0,1}) = \\
        &= i \, d^* d L (\gamma^{1,0} - \gamma^{0,1}). 
    \end{align*}
    Thus we obtained that $d L (\gamma^{1,0} - \gamma^{0,1})=0$. Therefore
    \[
    d^c \alpha = J^{-1} d L (\gamma^{1,0} - \gamma^{0,1}) =0,
    \]
    showing that $\alpha \in \H^3_{d + d^c} $. This completes the proof.
\end{proof}

\begin{cor}\label{cor:numerical:equality}
    If $(M,J,g,\omega)$ is an almost K\"ahler $4$-manifold, then
    \[
    h^3_{d+d^c} = h^3_{d+d^\Lambda}.
    \]
\end{cor}

In light of theorem \ref{thm:h3}, we ask what is the relation between $\H^k_{d+d^c}$ and $\H^k_{d+d^\Lambda}$ for $k \neq 3$. The next propositions answer the question.

\begin{proposition}\label{prop:inclusion:h1}
    Let $(M,J,g,\omega)$ be an almost K\"ahler $4$-manifold. Then 
    \[
    \H^2_{d+d^c} \subseteq \H^2_{d+d^\Lambda}.
    \]
    Furthermore, there exists an almost K\"ahler $4$-manifold such that
    \[
    \H^2_{d+d^c} \subsetneq \H^2_{d+d^\Lambda}.
    \]
\end{proposition}

\begin{proof}
    Let $\alpha \in \H^2_{d+d^c}$. Using the decomposition from theorem \ref{thm:decomposition}, we can write
    \[
    \alpha = c \, \omega + \gamma^{1,1} + \alpha^{2,0} + \alpha^{0,2},
    \]
    with $c$ constant, $\gamma^{1,1}$ primitive and anti-self-dual, and $\alpha^{2,0} + \alpha^{0,2}$ primitive and self-dual. We can compute that
    \[
    d^\Lambda \alpha = (d \Lambda - \Lambda d) \alpha = d c =0,
    \]
    and
    \[
    d d^\Lambda * \alpha = d d^\Lambda (c \omega - \gamma^{1,1} + \alpha^{2,0} + \alpha^{0,2}) = d d^\Lambda ( \alpha - 2 \gamma^{1,1} ) = - 2 d d^\Lambda \gamma^{1,1} =0,
    \]
    so that $\alpha \in \H^2_{d + d^\Lambda}$. For the second part of the proposition, we give an explicit example. On the Kodaira-Thurston manifold endowed with the almost K\"ahler structure from example \ref{ex:KT}, we have that $h^1_{d + d^c} = 2$ and $h^2_{d + d^c}=4$ (Proposition \ref{prop:KT:harmonic}), while $h^1_{d + d^\Lambda} = 3$ and $h^2_{d + d^\Lambda}=5$ \cite[Example 3.4]{TY12a}.
\end{proof}

\begin{proposition}
    Let $(M,J,g,\omega)$ be an almost K\"ahler $4$-manifold. Then
    \[
    \H^1_{d + d^c} \subseteq \H^1_{d + d^\Lambda}
    \]
    if and only if every $d$-closed $(1,0)$-form is also $d$-harmonic.
\end{proposition}
\begin{proof}
    Assume that every $d$-closed $(1,0)$-form is also $d$-harmonic and let $\alpha \in \H^1_{d + d^c}$. Then $d\alpha =0$ and $d^\Lambda \alpha =0$ for degree reasons. Furthermore, we have that $d d^\Lambda * \alpha = - d \Lambda d * \alpha = 0$ since the $(1,0)$ and $(0,1)$-degree parts of $\alpha$ are $d$-harmonic.\\
    Conversely, assume that $\H^1_{d + d^c} \subseteq \H^1_{d + d^\Lambda}$. We need to prove that if every $d$-closed $(1,0)$-form $\alpha^{1,0}$ is $(d d^\Lambda)^*$-closed, then it is also $d$-harmonic. We have that
    \[
    0 = d d^\Lambda * \alpha^{1,0} = - d \Lambda d * \alpha^{1,0} = -d \Lambda ( \bar \partial * \alpha^{1,0}).
    \]
    Setting $\bar \partial * \alpha^{1,0} = f \omega^2$, it follows that
    \[
    0 = - d \Lambda (f \omega^2) = - d(f \omega) = - df \wedge \omega.
    \]
    Since $L \colon A^1 \rightarrow A^3$ is an isomorphism, we have that $df =0$. Noting that
    \[
    d d^* \alpha^{1,0} = d * (f \omega^2) = df =0,
    \] 
    we conclude that $\alpha^{1,0}$ is $d$-harmonic.
\end{proof}
Example \ref{ex:sol3} shows that in general the inclusion can be strict.

We conclude the section showing that the number $h^1_{d + d^c}$ can be used to distinguish between almost complex structures compatible with the same symplectic form.

\begin{proposition}\label{prop:symplectic}
    There exists a symplectic $4$-manifold $(M,\omega)$ admitting two $\omega$-compatible almost K\"ahler structures $J_1$ and $J_2$ such that 
    \[
    h^1_{d+d^c}(J_1) \neq h^1_{d+d^c} (J_2).
    \]
    In particular, $h^1_{d+d^c}$ distinguishes between almost complex structures compatible with the same symplectic form.
\end{proposition}

\begin{proof}
    Let $M$ be the $4$-manifold $(\Gamma \backslash Sol(3)) \times \S^1$, where $Sol(3)$ is the only $3$-dimensional solvable, non-nilpotent Lie group and $\Gamma$ is a suitable lattice. Then $M$ admits a symplectic form $\omega$ and a family of $\omega$-compatible almost complex structures $J_t$, $t \in \R$, $\abs{t} < \epsilon$, such that $h^1_{d+d^c}(J_0) = 2$ and $h^1_{d+d^c} (J_t) = 0$ for $t \neq 0$, $\abs{t} < \epsilon$ (see example \ref{ex:sol3}).
\end{proof}

\section{Examples}\label{sec:examples}

In this section we give several examples of almost complex manifolds on which we compute the spaces of harmonic forms and the cohomologies we introduced. We will compute the number $h^1_{d+d^c}$ and show that a bigraded decomposition of Bott-Chern cohomology does not hold in general.

\begin{example}[The Kodaira-Thurston manifold]\label{ex:KT}

Consider the $3$-dimensional Heisenberg group
\[
\mathbb{H}_3 := \left \{ 
\begin{bmatrix}
1 & x & z \\
  & 1 & y \\
  &   & 1 \\
\end{bmatrix}: x,y,z \in \R \right \}.
\]
The Kodaira-Thurston manifold is the compact quotient
\[
M:= \mathbb{H}_3 / ( \mathbb{H}_3 \cap SL (3, \Z) ) \times \S^1.
\]
Denote by $t$ the coordinate on $\S^1$. A frame of left-invariant vector fields on $M$ is given by
\[
\left \{ e_1 = \partial_t, \, e_2 = \partial_x, \, e_3 = \partial_y + x \, \partial_z, \, e_4 = \partial_z \right \},
\]
and the dual co-frame of left-invariant $1$-forms is 
\[
\left \{ e^1 = dt, \, e^2 = dx , \, e^3 = dy, \, e^4 = dz - x \, dy \right \}.
\]
The only non-vanishing Lie bracket on vector fields is $[ e_2, e_3]=e_4$, thus the only non-vanishing differential is $de^4 = - e^{23}$. $M$ admits a symplectic form
\[
\omega = e^{12} + e^{34}.
\]
Consider the $\omega$-compatible almost complex structure $J$ given by
\begin{equation}\label{eq:std:acs:kt}
Je^1 = e^2, \quad J e^3 = e^4.
\end{equation}
A co-frame of $(1,0)$ forms is 
\[
\phi^1= e^1 + i e^2, \quad  \phi^2= e^3 + i e^4,
\]
with differentials
\[
d \phi^1 =0, \quad d \phi^2 = - \frac{1}{4} ( \phi^{12} + \phi^{1 \bar 2} - \phi^{\bar12} - \phi^{\bar 1 \bar 2}).
\]
Denote by $\xi_1$ and $\xi_2$ the dual frame of $(1,0)$-vector fields.\\
We first compute the spaces of $(d+d^c)$-harmonic forms.

\begin{proposition}\label{prop:KT:harmonic}
    The spaces of $(d+d^c)$-harmonic forms on the Kodaira-Thurston manifold endowed with the almost K\"ahler structure \eqref{eq:std:acs:kt} are the following
    \begin{align*}
        &\H^0_{d+d^c} = \C, \\
        &\H^1_{d+d^c} = \C \langle \phi^1, \, \phi^{\bar 1} \rangle, \\
        &\H^2_{d+d^c} = \C \langle \phi^{1 2} + \phi^{\bar 1 \bar 2}, \, \phi^{1 \bar 2} + \phi^{\bar 1 2}, \, \phi^{1 \bar 1}, \, \phi^{2 \bar 2} \rangle, \\
        &\H^3_{d+d^c} = \C \langle \phi^{12\bar 2}, \, \phi^{2 \bar 1 \bar 2}, \, \phi^{12 \bar 1} - \phi^{1 \bar 1 \bar 2} \rangle, \\
        &\H^4_{d+d^c} = \C \langle \phi^{12\bar 1\bar 2} \rangle.
    \end{align*}
\end{proposition}
\begin{proof}
    By proposition \ref{prop:metric:indep}, $\H^k_{d+d^c}$ is metric independent for $k\in \{0,1,4\}$. The claim for $k \in \{ 0,4 \}$ is immediate. For $k=1$, we need to establish which $1$-forms are both $d$-closed and $d^c$-closed. This is equivalent to find $d$-closed $(1,0)$-forms. Let $\alpha \in A^{1,0}$. Then
    \[
    \alpha = f \phi^1 + g \phi^2,
    \]
    with $f,g \in C^\infty(M)$. Writing explicitly the equation $d \alpha =0$ and separating the bidegree of the forms, we deduce that $f$ must be constant and $g=0$, hence $\H^1_{d+d^c} = \C \langle \phi^1, \phi^{\bar 1} \rangle.$\\
    Now we compute left-invariant, $(d+d^c)$-harmonic $2$-forms and we show that there is no other $(d+d^c)$-harmonic $2$-form. Let $\alpha \in A^2$. Assume that $\alpha$ is left-invariant, i.e., that 
    \[
    \alpha = a \phi^{12} + e \phi^{1 \bar 1} + f \phi^{1 \bar 2} + g \phi^{\bar 1 2} + h \phi^{2 \bar 2} + b \phi^{ \bar 1 \bar 2},
    \]
    with $a,b,e,f,g,h \in \C$. By computing separately the differential on the $(2,0)+(0,2)$-part and on the $(1,1)$-part of $\alpha$, we see that the constants must satisfy $a=b$ and $f=g$. Moreover, the condition $d d^c * \alpha =0$ is satisfied since all left-invariant $3$-forms on $M$ are $d$-closed. Thus we get
    \[
     \C \langle \phi^{1 2} + \phi^{\bar 1 \bar 2}, \phi^{1 \bar 2} + \phi^{\bar 1 2}, \phi^{1 \bar 1}, \phi^{2 \bar 2} \rangle \subseteq \H^2_{d+d^c}.
    \]
    Moreover, we also have that
    \[
     b_2 (M) = 4 = \dim_\C (\C \langle \phi^{1 2} + \phi^{\bar 1 \bar 2}, \phi^{1 \bar 2} + \phi^{\bar 1 2}, \phi^{1 \bar 1}, \phi^{2 \bar 2} \rangle) \le h^2_{d+d^c} \le b_2(M),
    \]
    by corollary \ref{cor:upper:bound}. This implies the equality of the spaces and that all $(d+d^c)$-harmonic $2$-forms are left-invariant. Finally, we know the space of $(d+d^c)$-harmonic $3$-forms thanks to theorem \ref{thm:h3} and the computations of Tseng and Yau \cite[section 3.4]{TY12a} rewritten in terms of complex forms.
\end{proof}

We proceed to show that the Bott-Chern cohomology of $(M,J)$ can be infinite dimensional. We already know that
\[
H^0_{d+d^c} \cong \C, \quad H^4_{d+d^c} \cong \C, \quad H^1_{d+d^c} = \C \langle \phi^1, \, \phi^{\bar 1} \rangle \cong \C^2,
\]
so that they are finite dimensional. We compute Bott-Chern cohomology on $2$-forms. Let 
    \[
    \alpha = a \phi^{12} + e \phi^{1 \bar 1} + f \phi^{1 \bar 2} + g \phi^{\bar 1 2} + h \phi^{2 \bar 2} + b \phi^{ \bar 1 \bar 2},
    \]
    with $a,b,e,f,g,h \in C^\infty(M)$, be a $2$-form. Direct computations show that $\alpha$ is $d$-closed and $d^c$-closed if and only if there exists functions $a,b,e,f,g,h$ satisfying the system
    \begin{equation}\label{eq:system}\tag{$\ast$}
        \begin{cases}
        \xi_{\bar 2} (a ) = \xi_{2} (b) =0, \\
        \frac{1}{4} (a-b) + \xi_{\bar 1} (a) =0, \\
        -\frac{1}{4} (a-b) + \xi_{1} (b) =0, \\
        \xi_1(h) = \xi_2 (f), \\
        \xi_{\bar 1} (h) = - \xi_{\bar 2} (g), \\
        \frac{1}{4} (g-f) - \xi_{2} (e) - \xi_1(g) =0 \\
        -\frac{1}{4} (g-f) + \xi_{\bar 2} (e) + \xi_{\bar 1}(f) =0. \\
        \end{cases}
    \end{equation}

    Thus, we can describe the Bott-Chern cohomology.

\begin{proposition}\label{prop:KT:cohomology}
    The $2^\text{nd}$ Bott-Chern cohomology group of the Kodaira-Thurston manifold endowed with the almost complex structure \eqref{eq:std:acs:kt} is
    \[
        H^2_{d+d^c} = \H \oplus \frac{\mathcal{I}}{\mathcal{J}},
    \]
    where
    \begin{align*}
    &\H = \{ a \phi^{12} + b \phi^{\bar 1\bar 2} :\eqref{eq:system} \text{ holds } \},\\
    &\mathcal{I} = \{ e \phi^{1 \bar 1} + f \phi^{1 \bar 2} + g \phi^{\bar 1 2} + h \phi^{2 \bar 2} : \eqref{eq:system} \text{ holds } \},
    \end{align*}
    and
    \[
    \mathcal{J} = \{ \xi_1 \xi_{\bar 1} (\theta) \phi^{ 1 \bar 1 }: \theta \in C^\infty(M), \, \xi_2(\theta)= \xi_{\bar 2}(\theta)=0 \}.
    \]
    Moreover, $H^2_{d+d^c}$ is infinite dimensional.
\end{proposition}

\begin{proof}
    The space of $d$-closed and $d^c$-closed $2$-forms is given by 
    \[
    \H \oplus \mathcal{I}.
    \]
    To compute Bott-Chern cohomology, we have to quotient by $d d^c \theta $, where $\theta \in C^\infty(M)$ is such that $(d d^c + d^c d) \theta=0$. By \eqref{eq:delta2}, $\theta$ must be $\partial^2$-closed and $\bar \partial^2$-closed. From the first condition, we have that 
    \[
    0 = \partial^2 \theta =  ( \xi_1 \xi_2 (\theta) - \xi_2 \xi_1(\theta) - \frac{1}{4} \xi_2(\theta)) \phi^{12} =- \frac{1}{4} \xi_{\bar 2} (\theta) \phi^{12},
    \]
    where the last equality follows from the commutator relation
    \[
    [ \xi_1, \xi_2] = \frac{1}{4} (\xi_2 - \xi_{\bar 2}).
    \]
    This implies that $\xi_{\bar 2} (\theta)=0$. Similarly, we have that $\xi_{2} (\theta)=0$. Finally, we can compute
    \[
    d d^c \theta = 2 i \xi_1 \xi_{\bar 1} (\theta) \phi^{1 \bar 1},
    \]
    showing that we have to quotient by $\mathcal{J}$.\\
    We now prove that $H^2_{d+d^c}$ contains an infinite dimensional subspace. Consider the subspace 
    \[
    \mathcal{S} = \{ h \phi^{2 \bar 2}: \xi_1(h) =0, \, \xi_{\bar 1}(h)=0 \}.
    \]
    By the expression of $H^2_{d+d^c}$, we have that $\mathcal{S} \subset H^2_{d+d^c}$, and $\mathcal{S}$ is infinite dimensional because it strictly contains the family of functions $\{ cos (2 \pi n y) \}_{n \in \N}$.
\end{proof}

\end{example}

\begin{example}[There is no bigraded splitting for Bott-Chern cohomology in degree $k=2$]\label{ex:no:decomposition}
    Let $(M,J)$ be the Kodaira-Thurston endowed with the almost K\"ahler structure of example \ref{ex:KT}. We know that 
    \[
        H^2_{d+d^c} = \H \oplus \frac{\mathcal{I}}{\mathcal{J}}.
    \]
    Moreover $\H \neq \{ 0 \}$ because it contains the $2$-forms $c (\phi^{12} + \phi^{\bar 1 \bar 2})$ for any $c \in \C$.\\
    We show that there is no bigraded decomposition by proving that $H^{2,0}_{d + d^c} = H^{0,2}_{d + d^c} = \{ 0 \}$. Let $ f \phi^{12} \in A^{2,0}$ be a $d$-closed $(2,0)$-form. In particular, it must be
    \[
    0 = \bar \mu (f \phi^{12}) = - f \phi^1 \wedge \bar \mu \phi^2 = \frac{1}{4} f \phi^{1 \bar 1 \bar 2},
    \] 
    so that $f =0$ and $H^{2,0}_{d + d^c} = \{ 0 \}$. By conjugation, we also have  $H^{0,2}_{d + d^c} = \{ 0 \}$.
\end{example}

\begin{example}[The number $h^1_{d + d^c}$ distinguishes between almost K\"ahler structures]\label{ex:sol3}
Let $Sol(3)$ be the only $3$-dimensional solvable, non-nilpotent Lie group. Let $M$ be the $4$-manifold obtained as a quotient of $Sol(3) \times \R$ by a co-compact lattice. A co-frame of left-invariant forms on $M$ is given by $\{ e^1, e^2, e^3, e^4 \}$ with structure equations
\[
d e^3 = -e^{13}, \quad de^4 = e^{14}.
\]
The manifold $M$ admits no complex structure (cf.\ \cite{Boc16}), but admits a symplectic structure
\[
\omega = e^{12} + e^{34}.
\]
For $t \in \R$ small enough, let $J_t$ be the family of $\omega$-compatible almost K\"ahler structures defined by the $(1,0)$-forms (see \cite[Section 6.1]{FT10})
\begin{align*}
    &\phi^1_t = e^1 + i \Big( \frac{1 + t^2}{1 -t^2} e^2 - \frac{2t}{1-t^2} e^4 \Big), \\
    &\phi^2_t = e^3 + i \Big( \frac{2t}{1 -t^2} e^2 + \frac{1 + t^2}{1-t^2} e^4 \Big).
\end{align*}
Direct computations show that real forms are expressed in terms of complex forms as
\[
e^1 = \frac{1}{2} ( \phi^1_t + \phi^{\bar 1}_t), \quad  e^3 = \frac{1}{2} ( \phi^2_t + \phi^{\bar 2}_t),
\]
and
\[
e^4 = - \frac{i}{2} \, \frac{(1 - t^2)(1+t^2)}{1 + 6 t^2 + t^4} \Big( \phi^2_t - \phi^{\bar 2}_t - \frac{2t}{1 + t^2} (\phi^1_t - \phi^{\bar 1}_t) \Big).
\]
Consequently, the differentials of the $(1,0)$-co-frame are
\[
d \phi^1_t = - \frac{1}{2} \, \frac{t (1+t^2)}{1 + 6 t^2 + t^4} \Big( \phi^{12}_t - \phi^{1 \bar 2}_t + \phi^{\bar 1 2}_t -  \phi^{\bar 1 \bar 2}_t  + \frac{4t}{1+t^2}\phi^{1 \bar 1}_t \Big)
\]
and
{\small
\[
d \phi^2_t = - \frac{1}{4}  (\phi^{12}_t + \phi^{1 \bar 2}_t + \phi^{\bar 1 2}_t +  \phi^{\bar 1 \bar 2}_t) +
\frac{1}{4} \, \frac{(1 + t^2)^2}{1 + 6 t^2 + t^4} \Big( \phi^{12}_t - \phi^{1 \bar 2}_t + \phi^{\bar 1 2}_t -  \phi^{\bar 1 \bar 2}_t  + \frac{4t}{1+t^2}\phi^{1 \bar 1}_t \Big). 
\]
}
\end{example}
Recall that $h^1_{d + d^c} (J_t) = 2 \dim_\C (\ker d \cap A^{1,0}_t)$. Let
\[
\alpha = f  \phi^1_t + g  \phi^2_t
\]
be a $d$-closed $(1,0)$-form. In particular, we have that 
\[
0 = \bar \mu \alpha =  \Big( \frac{1}{2} \, \frac{t (1+t^2)}{1 + 6 t^2 + t^4}f - \frac{1}{4} \, g \Big(1 + \frac{(1 + t^2)^2}{1 + 6 t^2 + t^4} \Big) \Big) \phi^{\bar 1\bar 2}_t, 
\]
implying that
\begin{equation}\label{eq:g:multiple}
   g= c_t f, \quad \text{with} \,  c_t = \frac{t (1+t^2)}{1 + 4 t^2 + t^4}.
\end{equation}
Taking the coefficients of $d \alpha$ corresponding to $\phi^{\bar 1 2}_t$ and $\phi^{1 \bar 1}_t$, we get the following equations
\begin{equation}\label{eq:absurd}
    \begin{cases}
        \xi_{\bar 1} (g) - \frac{1}{2} \frac{t (1+t^2)}{1 + 6 t^2 + t^4}f - \frac{1}{4}g \Big(1 - \frac{(1 + t^2)^2}{1 + 6 t^2 + t^4} \Big) =0, \\
        - \xi_{\bar 1}(f) - \frac{2 t^2}{1 + 6 t^2 + t^4}f + \frac{t(1+t^2)}{1 + 6 t^2 + t^4}g =0.
    \end{cases}
\end{equation}
Combining \eqref{eq:g:multiple} and \eqref{eq:absurd}, we deduce that:
\begin{itemize}
    \item for $t \neq 0$ small enough, it must be $f = g =0$. Thus $\alpha =0$ and $h^1_{d + d^c} (J_t) = 0$;

    \item for $t=0$, we have $g =0$. Using the equation $d \alpha =0$ it is not hard to see that $f$ must be constant, $\alpha = f \phi^1_0$ and $h^1_{d + d^c} (J_t) = 2$.
\end{itemize}
As a consequence, we have that $h^1_{d + d^c}$ can be used to distinguish between almost complex structures compatible with the same symplectic form (see Proposition \ref{prop:symplectic}).

{\small
\printbibliography
}

\end{document}